\documentclass[12pt,reqno]{amsart}
\usepackage{amssymb,latexsym, amsmath, amsxtra, mathrsfs, bm}
\usepackage{bookmark}
\usepackage{enumerate}
\usepackage{hyperref}
\usepackage{tikz}
\usetikzlibrary{patterns}
\usepackage[top=2.5cm, bottom=2.5cm, inner=2.5cm, outer=2.5cm, marginparwidth=2cm]{geometry}
\usepackage{mathtools}
\usepackage{graphicx}
\usepackage{subcaption}
\newtheorem{thmy}{Theorem}
\newenvironment{thmx}{\stepcounter{thm}\begin{thmy}}{\end{thmy}}

\makeatletter

\newcommand{\Rmnum}[1]{\expandafter\@slowromancap\romannumeral #1@}
\makeatother

\theoremstyle{plain}
\newtheorem{theorem}{Theorem}[section]
\newtheorem{cor}[theorem]{Corollary}
\newtheorem{prop}[theorem]{Proposition}
\newtheorem{lemma}[theorem]{Lemma}
\newtheorem{definition}[theorem]{Definition}
\newtheorem{remark}[theorem]{Remark}
\newtheorem{example}[theorem]{Example}
\newtheorem{proposition}[theorem]{Proposition}

\newtheorem*{claim}{Claim}

\begin{document}
\title[]{On fast Lyapunov spectra for Markov-R\'{e}nyi maps}

\author{Lulu Fang}

\author{Carlos Gustavo Moreira}

\author{Zhichao Wang}

\author{Yiwei Zhang}


\begin{abstract}
  In this paper, we study the multifractal analysis for Markov-R\'{e}nyi maps, which form a canonical class of piecewise differentiable interval maps, with countably many branches and may contain a parabolic fixed point simultaneously, and do not assume any distortion hypotheses. We develop a geometric approach, independent of thermodynamic formalism, to study the fast Lyapunov spectrum for Markov-R\'{e}nyi maps. Our study can be regarded as a refinement of the Lyapunov spectrum at infinity. We demonstrate that the fast Lyapunov spectrum is a piecewise constant function, possibly exhibiting a discontinuity at infinity. Our results extend the works in \cite[Theorem 1.1]{FLWW13}, \cite[Theorem 1.2]{LR}, and \cite[Theorem 1.2]{FSW} from the Gauss map to arbitrary Markov-R\'{e}nyi maps, and highlight several intrinsic differences between the fast Lyapunov spectrum and the classical Lyapunov spectrum. Moreover, we establish the upper and lower fast Lyapunov spectra for Markov-R\'{e}nyi maps.

\end{abstract}

\maketitle

\tableofcontents

\section{Introduction}
\subsection{Backgrounds}
Dynamical systems with strong hyperbolicity typically exhibit a high degree of mixing, resulting in the separation of nearby orbits. In such systems, the intricate orbit structure necessitates a quantitative measure of their asymptotic separation rates.
This behavior is precisely captured by Lyapunov exponents, which describe the exponential rate at which infinitesimally close orbits diverge over time.

Let $I\coloneqq[0,1]$, and let $T: I\to I$ be a piecewise differentiable interval map. The \emph{Lyapunov exponent} of $T$ at a point $x\in I$ is defined as
\begin{equation}\label{equ:Lyapunov exponent}
\lambda(x)\coloneqq\lim_{n \to \infty}\frac{1}{n}\log |(T^n)^\prime(x)|,
\end{equation}
whenever the limit exists. By Birkhoff's ergodic theorem, if $\mu$ is an ergodic $T$-invariant probability measure such that $\int \log|T^\prime|d\mu <\infty$, then $\lambda(x) = \int \log|T^\prime|d\mu$ for $\mu$-almost every $x\in I$. However, in general, the Lyapunov exponent may attain a continuum of values, forming an entire interval. This motivates the study of the \emph{level sets}
\begin{equation}\label{equ_levelset}
J(\alpha)\coloneqq \left\{x\in I: \lambda(x)=\alpha\right\},
\end{equation}
and the associated \emph{Lyapunov spectrum}
\begin{equation}\label{equ_Lyapunovspectrum}
L(\alpha)\coloneqq \dim_{\rm H}J(\alpha),
\end{equation}
where $\dim_{\rm H}$ denotes the Hausdorff dimension (see \cite{Fal} for the definition).

The Lyapunov spectrum has been extensively studied for several important classes of piecewise differentiable maps.
For finitely branched conformal expanding maps, Weiss \cite{Wei} established a connection between the Lyapunov exponent and the pointwise dimension of a Gibbs measure. Applying the multifractal analysis of pointwise dimensions developed by Pesin and Weiss \cite{PW97}, he showed that the Lyapunov spectrum is real-analytic and defined on a bounded interval. Similar analyses have been extended to non-uniformly expanding systems, such as the Manneville-Pomeau map, where the Lyapunov spectrum remains defined on a bounded interval but may exhibit points of non-analyticity (see \cite{PW}). Gelfert and Rams \cite{GR} studied the Lyapunov spectrum for a class of interval maps with parabolic periodic points.
Further results for finite-branch Markov interval maps can be found in \cite{BS, FLW, JJOP, Nak, PS, TV03}.

In contrast, when $T$ has countably many branches, the Lyapunov spectrum exhibits behaviors significantly different from those observed in the finite-branch case. Pollicott and Weiss \cite{PW} studied the Lyapunov spectrum for the Gauss map, which is a  uniformly expanding map with countably many branches. They demonstrated that the Lyapunov spectrum is real analytic, but it has an unbounded domain. Subsequent works by Kesseb\"{o}hmer and Stratmann \cite{KS07}, and Fan et al. \cite{FLWW09} provided a complete characterization of the Lyapunov spectrum for the Gauss map. For non-uniformly expanding maps, Iommi \cite{Iommi} showed that the Lyapunov spectrum has an unbounded domain, and there might exist a non-differentiability  in the domain. For additional multifractal analysis of the Lyapunov spectrum in Markov interval maps with countably many branches, including extensions and related results, we refer to \cite{FJLR,IJ,JT,KMS,MU,Rush} and the references therein.

In summary, for interval maps with countably many branches, the domain of the Lyapunov spectrum is often unbounded, reflecting the possibility of infinite Lyapunov exponents. This indicates that, for certain points, the divergence of nearby orbits occurs at a rate faster than any exponential scale. Hence, the  Lyapunov exponent becomes insufficient to describe such rapid behaviors.

To overcome this limitation and provide a finer understanding of orbit separation beyond exponential growth, we introduce the concept of the fast Lyapunov exponent. This quantity is designed to capture the super-exponential separation of orbits and enables a refined multifractal analysis of the corresponding level sets in terms of Hausdorff dimension.

Let us now define the dynamical systems under consideration and present the main results of this paper.

\subsection{Markov-R\'{e}nyi maps}\label{subsec_MR}

%

Inspired by the works of Iommi \cite{Iommi}, Pollicott and Weiss \cite{PW}, and Sarig \cite{Sar09}, we propose the following definition of the Markov-R\'{e}nyi map.

\begin{definition}\label{def:MR}
Let $\{I_n\}_{n \in \mathbb N}$ be a collection of disjoint open subintervals of $I=[0,1]$.
A piecewise differentiable interval map $T: \cup_{n \in \mathbb N}\overline{I}_n \to I$ is called a \emph{Markov-R\'{e}nyi map} associated with $\{I_n\}_{n \in \mathbb N}$ if the following conditions hold:
\begin{itemize}
    \item[(1).] for every $n\in \mathbb{N}$, the intervals $I_n$ and $I_{n+1}$ share a common endpoint i.e., $\partial I_{n}\cap\partial I_{n+1}$ is a singleton;
     \item[(2).] for every $n\in \mathbb{N}$, the restriction $T|_{I_n}$ extends to a $C^{1}$ diffeomorphism $T_n$ on an open neighborhood of $\overline{I}_n$;
    \item[(3).] there exists a point $p \in \cup_{n \in \mathbb N}\overline{I}_n$ and an integer $m\in \mathbb{N}$ such that $T(p)=p$, $|T^\prime(p)| \geq 1$, and $|(T^m)^\prime(x)|>1$ for all $x\in \cup_{n \in \mathbb N}\overline{I}_n\setminus \{p\}$;
    \item[(4).] $\overline{T(I_n)}=[0,1]$ for every $n\in \mathbb{N}$;
    \item[(5).] there exist two constants $\gamma>1$ and $C>1$ such that for all $n \in \mathbb{N}$ and $x \in \overline{I}_n$, we have
    \begin{equation*}
        C^{-1}n^{\gamma}\leq |T^\prime_n(x)|\leq Cn^{\gamma}.
    \end{equation*}
\end{itemize}
\end{definition}

\begin{remark}
In comparison with the definition used in \cite[Section 2]{Iommi}, Definition \ref{def:MR} does not require the Markov-R\'{e}nyi map to satisfy any distortion hypotheses. In particular, Definition \ref{def:MR} includes the Markov-R\'{e}nyi maps without tempered distortion condition \footnote{The tempered distortion condition was introduced in the works of Gelfert and Rams \cite{GR1,GR}, and is treated as one of the mild distortion conditions. It states that there exists a sequence of positive numbers $\{\rho_{n}\}_{n\in\mathbb{N}}$ decreasing to zero, such that
\begin{equation}\label{equ:tempered}
\sup_{(i_{1},\dots,i_{n})\in\mathbb{N}^{n}}\sup_{x,y\in I_{n}(i_{1},\dots,i_{n})}\frac{|(T^{n})'(x)|}{|(T^{n})'(y)|}\leq \exp(n\rho_{n}), \quad \forall n\in\mathbb{N},
\end{equation}
where $I_{n}(i_{1},\dots,i_{n})$ denotes the cylinder set of order $n$ defined in \eqref{def:Cylinder}. The tempered distortion condition can be yielded in many systems, such as the piecewise differentiable interval maps that satisfy the R\'{e}nyi condition, see \cite{PW}.}.
\end{remark}

A point $p\in \cup_{n \in \mathbb N}\overline{I}_n$ such that $T(p)=p$ and $|T^\prime(p)| = 1$ is called a \emph{parabolic fixed point}. We say $T$ is \emph{uniformly expanding} if there exist $m\in \mathbb{N}$ and $c>1$ such that $|(T^m)^\prime(x)|\geq c$ for all $x\in \cup_{n \in \mathbb N}\overline{I}_n$. Of course, if $T$ is uniformly expanding, then it does not contain any parabolic fixed points.

There are two important examples of Markov-R\'{e}nyi maps in the study of (backward) continued fractions, which exhibit distinct behaviors in their Lyapunov spectra.

\begin{example}[The Gauss map]
The Gauss map $G: [0,1] \to [0,1]$ is defined by $G(0)\coloneqq 0$ and
\[
G(x)\coloneqq
    \frac{1}{x}-  \left\lfloor\frac{1}{x}\right\rfloor,\quad \forall x\in (0,1],
\]
where $\lfloor \cdot\rfloor$ denotes the integer part of a number. It is a uniformly expanding Markov-R\'{e}nyi map.
In this case, $I_n=(\frac{1}{n+1}, \frac{1}{n})$ for all $n \in \mathbb N$, and the parameter $\gamma=2$ in Hypothesis $(5)$ of Definition \ref{def:MR}.
\end{example}

\begin{example}[The R\'{e}nyi map]
The R\'{e}nyi map $R: [0,1] \to [0,1]$ is defined by $R(1)\coloneqq 0$ and
\[
R(x)\coloneqq \frac{1}{1-x}-  \left\lfloor\frac{1}{1-x}\right\rfloor,\quad \forall x\in [0,1).
\]
It is a Markov-R\'{e}nyi map with a parabolic fixed point at zero. In this case, $I_n=(\frac{n-1}{n}, \frac{n}{n+1})$ for all $n \in \mathbb N$, and the parameter $\gamma=2$ in Hypothesis $(5)$ of Definition \ref{def:MR}.

In contrast to the Gauss map, the R\'{e}nyi map has complexity associated with possessing both the existence of a parabolic fixed point and countable many branches. It is the interplay between these two features that makes $R(x)$ admitting independent interests (see \cite{AF,BL,GH,Renyi}).
\end{example}

Under the assumption that the Markov-R\'{e}nyi map $T$ satisfies the tempered distortion property (as in \eqref{equ:tempered}), the Lyapunov spectrum of Markov-R\'{e}nyi maps was established by Iommi \cite{Iommi}, who characterized the spectrum in terms of the Legendre-Fenchel transform of the geometric pressure function.

\begin{theorem}[{\cite[Theorem 4.1]{Iommi}}]\label{IomLS}
Let $T$ be a Markov-R\'{e}nyi map satisfying the tempered distortion property. Then the domain of the Lyapunov spectrum $L(\cdot)$ is an unbounded subinterval of $[0,\infty)$ and
$$
L(\alpha)=\frac{1}{\alpha} \inf_{t \in \mathbb R}\big\{P(t) +t\alpha\big\},
$$
where $P(t)$ denotes the geometric pressure function of $T$, i.e., the topological pressure for $T$ and the potential $-t\log|T'|$.
\end{theorem}

A detailed analysis of the geometric pressure function $P(\cdot)$, including its regularity, monotonicity, and convexity properties, is provided in \cite[Section 3]{Iommi}. Based on the behavior of the geometric pressure function and Theorem \ref{IomLS}, we can derive several properties of the Lyapunov spectrum.
In particular, for the Gauss map, the Lyapunov spectrum is defined on the unbounded interval $[2\log ((1+\sqrt{5})/2),\infty)$, and attains a unique maximum at $\pi^2/(6\log 2)$. Moreover, it is real-analytic, strictly increasing on the interval $[2\log ((1+\sqrt{5})/2), \pi^2/(6\log 2)]$, and strictly decreasing on $[\pi^2/(6\log 2),\infty)$. For the R\'{e}nyi map, the domain of the Lyapunov spectrum is the interval $[0,\infty)$. In this case, the Lyapunov spectrum is real-analytic, strictly decreasing on $[0,\infty)$, and has a unique maximum at zero.

The graphs of the Lyapunov spectra for $G$ and $R$ are shown in Figure \ref{GRmaps}.
\begin{figure}[htbp]
  \centering
  \begin{subfigure}[b]{0.45\textwidth}
    \includegraphics[width=\textwidth]{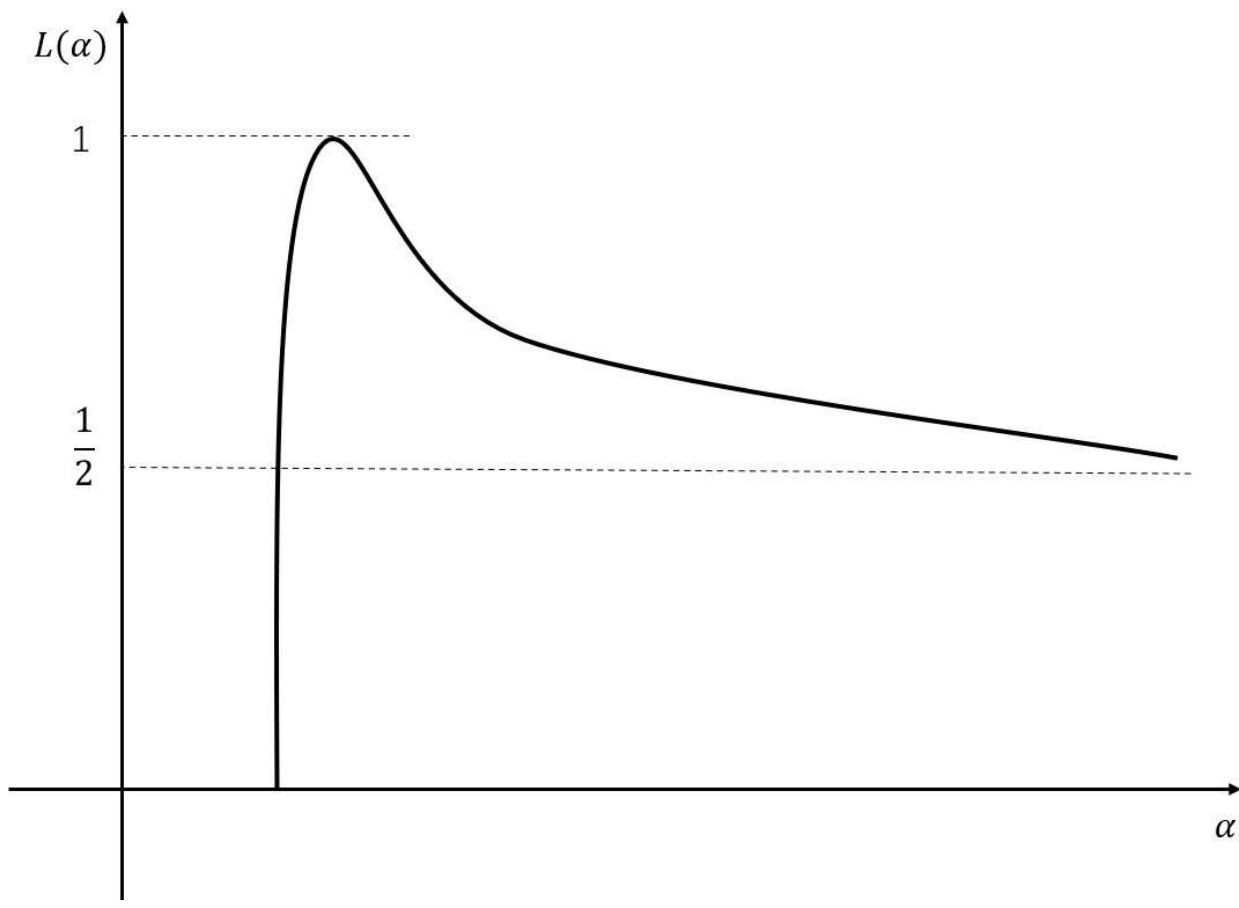}
  \end{subfigure}
  \hfill
  \begin{subfigure}[b]{0.45\textwidth}
    \includegraphics[width=\textwidth]{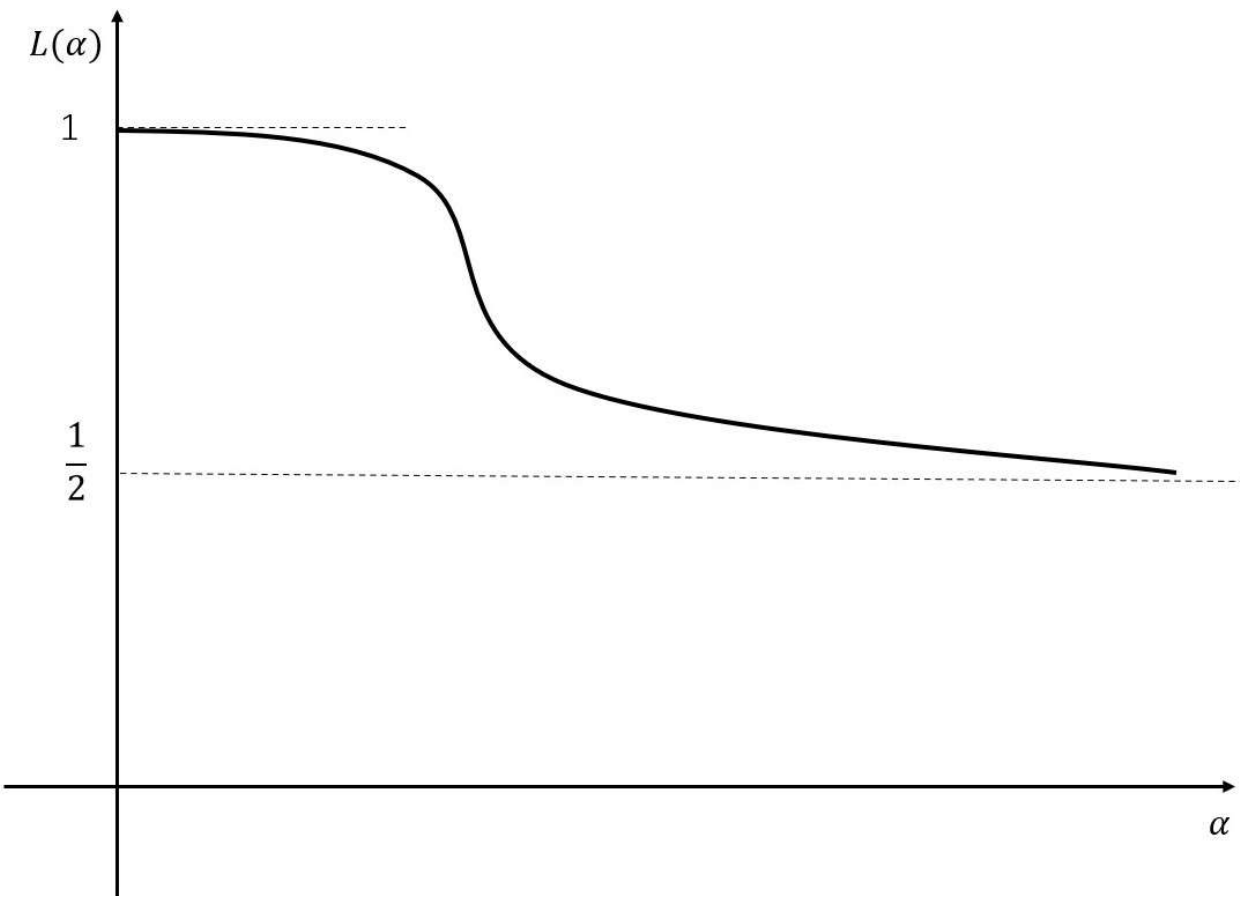}
  \end{subfigure}
  \caption{Left: Lyapunov spectrum of $G$ \quad Right: Lyapunov spectrum of $R$}\label{GRmaps}
\end{figure}

\noindent It is evident that the Lyapunov spectra of $G$ and $R$ exhibit horizontal asymptotes at $1/2$. In fact, by Hypothesis (5) of Definition \ref{def:MR}, we obtain
\begin{equation}\label{PB}
P(t) \geq \log \sum^\infty_{n=1} n^{-t\gamma},
\end{equation}
which implies that $P(t)$ has a logarithmic singularity at $t=1/\gamma$. From Theorem \ref{IomLS}, we conclude that
$$
\lim_{\alpha \to \infty}L(\alpha)=\frac{1}{\gamma}.
$$

This naturally leads to the following question:
\begin{itemize}
    \item[(Q1)] What is the value of $L(\infty)$, and is the Lyapunov spectrum continuous at infinity?
    \end{itemize}
An affirmative answer of $(Q1)$ is known in the case of the Gauss map (see \cite[Theorem 1.3]{FLWW09}) and the uniformly expanding Markov-R\'{e}nyi map (see \cite[Theorem 7.1]{IJ}) satisfying the R\'{e}nyi condition. We extend these results to arbitrary Markov-R\'{e}nyi maps.

\begin{proposition}\label{LSinfty}
Let $T$ be a Markov-R\'{e}nyi map. Then
$$
L(\infty)=\frac{1}{\gamma}.
$$
\end{proposition}

Proposition \ref{LSinfty} indicates that there are uncountably many real numbers with infinite Lyapunov exponent, which raises the question of conducting a more detailed analysis of such numbers. Accordingly, we pose the following question:
\begin{itemize}
    \item[(Q2)] Can the spectrum at Lyapunov exponent infinity be further refined? If so, how does it differ from the Lyapunov spectrum?
\end{itemize}

To address $(Q2)$, we analyze the behavior of the quantity $\log |(T^n)^\prime(x)|$, which diverges at a rate faster than linear.
This motivates the introduction of the fast Lyapunov exponent and the associated fast Lyapunov spectrum.
We remark that the fast Lyapunov exponent was first introduced by Fan et al. \cite{FLWW09,FLWW13} for the Gauss map, and further studied by Liao and Rams \cite{LR} and Fang, Shang and Wu \cite{FSW}.

\subsection{Statement of main results}
Let $T$ be a Markov-R\'{e}nyi map with repeller $\Lambda$ (see Section 2 for its definition), and let $\psi:\mathbb{N}\to\mathbb{R}_{>0}$ be a \emph{scaling function} such that $\psi(n)/n \to \infty$ as $n\to \infty$.

Analogously to \eqref{equ:Lyapunov exponent}, \eqref{equ_levelset} and \eqref{equ_Lyapunovspectrum}, we define the \emph{fast Lyapunov exponent} of $T$ (with respect to $\psi$) at a point $x\in \Lambda$ as
\begin{equation*}\label{equ_fastlypunovexponent}
\lambda_{\psi}(x)\coloneqq \lim_{n \to \infty}\frac{1}{\psi(n)}\log |(T^n)^\prime(x)|,
\end{equation*}
whenever the limit exists. For any $\alpha\in\mathbb{R}_{\geq 0} \cup\{\infty\}$, let $J_{\psi}(\alpha):=\left\{x\in \Lambda: \lambda_{\psi}(x) =\alpha\right\}$,
and define the \emph{fast Lyapunov spectrum} (with respect to $\psi$) as
\begin{equation*}\label{equ_fastLspectrum}
F_{\psi}(\alpha)\coloneqq \dim_{\rm H}J_{\psi}(\alpha).
\end{equation*}

Now, we are ready to present the main results of this paper. We say that two functions $f,g: \mathbb{N}\to\mathbb{R}_{>0}$ are \emph{equivalent} if $f(n)/g(n)\to 1$ as $n\to \infty$. Let
\begin{equation*}\label{equ_constants}
    \beta_\psi \coloneqq \limsup_{n\to \infty }
    \frac{\psi(n+1)}{\psi(n)} \quad \text{and}\quad
    B_\psi\coloneqq
    \limsup_{n\to \infty}\sqrt[n]{\psi(n)}.
\end{equation*}

The fast Lyapunov spectrum for Markov-R\'{e}nyi maps is described below.

\begin{thmx}[The fast Lyapunov spectrum]\label{t.Fast_Lyapunov}
Let $\psi:\mathbb{N}\to\mathbb{R}_{>0}$ be a function satisfying $\psi(n)/n \to \infty$ as $n\to \infty$. Then:
\begin{itemize}
  \item[(i)] when $\alpha =0$, the level set $J_\psi(0)$ has Lebesgue measure $\sum_{n \in \mathbb N}|I_n|$. In particular, $F_{\psi}(0)=1$;
  \item[(ii)] when $0<\alpha <\infty$, the level set $J_{\psi}(\alpha)$ is nonempty if and only if $\psi$ is equivalent to an increasing function. In this case,
\begin{equation*}\label{e.value_fast}
                F_{\psi}(\alpha)=
                \frac{1}{(\gamma-1)\beta_\psi+1};
\end{equation*}
\item[(iii)] when $\alpha =\infty$, we have
\begin{equation*}\label{equ:boudarypoints}
                F_{\psi}(\infty)=\frac{1}{(\gamma-1)B_\psi+1}.
\end{equation*}
\end{itemize}
\end{thmx}

Next, we investigate the upper and lower fast Lyapunov spectra of $T$. Specifically, for any $x\in \Lambda$, let
\begin{equation}\label{equ:upperfastLE}
\overline{\lambda}_{\psi}(x):=\limsup_{n \to \infty}\frac{1}{\psi(n)}\log |(T^n)^\prime(x)|,
\end{equation}
and let $\underline{\lambda}_{\psi}(x)$ be defined analogously by replacing the limit superior in \eqref{equ:upperfastLE} with the limit inferior.
For any $\alpha\in\mathbb{R}_{\geq 0} \cup\{\infty\}$, define
$$
\overline{J}_{\psi}(\alpha)\coloneqq\left\{x\in \Lambda: \overline{\lambda}_{\psi}(x) =\alpha\right\}\quad \text{and}\quad \underline{J}_{\psi}(\alpha)\coloneqq\left\{x\in \Lambda: \underline{\lambda}_{\psi}(x) =\alpha\right\},
$$
and call
\begin{equation*}\label{equ:u&lFLE}
\overline{F}_{\psi}(\alpha)\coloneqq\dim_{\rm H}\overline{J}_{\psi}(\alpha) \quad \text{and}\quad \underline{F}_{\psi}(\alpha)\coloneqq\dim_{\rm H}\underline{J}_{\psi}(\alpha)
\end{equation*}
the \emph{upper} and \emph{lower fast Lyapunov spectra} of $T$, respectively.

To describe the upper and lower fast Lyapunov spectra, let
\begin{equation*}\label{equ_constantsb}
   b_\psi \coloneqq \liminf_{n\to \infty}\sqrt[n]{\psi(n)}.
\end{equation*}

\begin{thmy}[The upper and lower fast Lyapunov spectra]\label{t.Upper_and_lower_fast_Lyapunov_spectrum}
Let $\psi:\mathbb{N}\to\mathbb{R}_{>0}$ be a function satisfying $\psi(n)/n \to \infty$ as $n\to \infty$. Then:
\begin{itemize}
  \item[(i)] when $\alpha =0$, both $\overline{J}_{\psi}(0)$ and $\underline{J}_{\psi}(0)$ have Lebesgue measure $\sum_{n \in \mathbb N}|I_n|$. In particular, $\overline{F}_{\psi}(0)=\underline{F}_{\psi}(0)=1$;
  \item[(ii)] when $0<\alpha \leq \infty$, we have
\begin{equation*}\label{equ:upperandlower1}
\overline{F}_{\psi}(\alpha)=\frac{1}{(\gamma-1)b_\psi+1}
\quad \text{and} \quad
\underline{F}_{\psi}(\alpha)=\frac{1}{(\gamma-1)B_\psi+1}.
\end{equation*}
\end{itemize}
\end{thmy}

In view of Theorems \ref{IomLS}, \ref{t.Fast_Lyapunov}, and \ref{t.Upper_and_lower_fast_Lyapunov_spectrum}, we identify two key differences between the fast Lyapunov spectrum and the Lyapunov spectrum.
\begin{itemize}
\item[(i)] The (upper and lower) fast Lyapunov spectrum is always a constant function on $\mathbb{R}_{>0} \cup\{\infty\}$ for every Markov-R\'{e}nyi map. Such constants depend merely on two features. Specifically, the dynamical system $T$ contributes through the parameter $\gamma$ in Hypothesis (5) of Definition \ref{def:MR}. The scaling function $\psi$ contributes through $\beta_\psi$ and $B_\psi$ in the fast Lyapunov spectrum, $b_\psi$ in the upper fast Lyapunov spectrum, and $B_\psi$ in the lower fast Lyapunov spectrum. While the Lyapunov spectrum described in Theorem \ref{IomLS}, is varied according to different Markov-R\'{e}nyi maps.
\item[(ii)] The existence of a parabolic fixed point of $T$ has no impact on the (upper and lower) fast Lyapunov spectrum, while the existence of a parabolic fixed point of $T$ plays a significant role on the Lyapunov spectrum.
\end{itemize}
On the other hand, the fast Lyapunov spectrum is discontinuous at zero, and it is continuous at infinity if and only if $\beta_\psi =B_\psi$. In contrast, both the upper and lower fast Lyapunov spectra exhibit a single discontinuity at zero. Moreover, note that $b_\psi\leq B_\psi\leq \beta_\psi$. This implies that the three associated fast Lyapunov spectra can be all distinct. For example,
let $\psi(1):=1$ and
$$
\psi(n):=
\left\{
  \begin{array}{ll}
    (\frac{5}{3})^{k-1}4^{n-(1!+3!+\cdots+(2k-1)!)}3^{1!+3!+\cdots+(2k-1)!}, & \hbox{$n_{2k-1}<n\leq n_{2k}$;} \vspace{0.15cm}\\
    (\frac{5}{3})^{k}4^{2!+4!+\cdots+(2k)!}3^{n- (2!+4!+\cdots+(2k)!)},& \hbox{$n_{2k}<n \leq n_{2k+1}$,}
  \end{array}
\right.
$$
where $n_k:=1!+2!+\cdots+k!$. Then $b_\psi=3, B_\psi=4$, and $\beta_\psi=5$; see \cite{FSW} for further details.
This phenomenon was first observed by Liao and Rams \cite{LR} for the Gauss map. We extend their contribution to arbitrary Markov-R\'{e}nyi maps.

It is worth remarking that the thermodynamic formalism approach employed in \cite{Iommi} for the Lyapunov spectrum is not applicable in proving Theorems \ref{t.Fast_Lyapunov} and \ref{t.Upper_and_lower_fast_Lyapunov_spectrum}. Let us explain as follows: as indicated in \eqref{PB}, the geometric pressure function $P(t)$ exhibits a logarithmic singularity at $t=1/\gamma$. This singularity implies that the pressure function fails to reflect the properties of sets whose Hausdorff dimension is smaller than $1/\gamma$. According to Proposition \ref{LSinfty}, the (upper and lower) fast Lyapunov spectrum is bounded above by $L(\infty)=1/\gamma$. Therefore, the thermodynamic formalism \textbf{cannot be directly employed} to characterize the (upper and lower) fast Lyapunov spectrum in this regime. The key feature of our approach to proving Theorems \ref{t.Fast_Lyapunov} and \ref{t.Upper_and_lower_fast_Lyapunov_spectrum} lies in the symbolic model for the Markov-R\'{e}nyi map and the distribution of the corresponding symbolic codings, which is independent of thermodynamic formalism.

Finally, we emphasize that Hypothesis $(1)$ of Definition \ref{def:MR}, concerning the collection $\{I_n\}_{n\in \mathbb N}$ induces a natural ordering, either from left to right or from right to left, is essential in establishing Theorems \ref{t.Fast_Lyapunov} and \ref{t.Upper_and_lower_fast_Lyapunov_spectrum}. Indeed, the following are two examples, examining the failures of Theorem \ref{t.Fast_Lyapunov}, and \ref{t.Upper_and_lower_fast_Lyapunov_spectrum} from two different perspectives, dissatisfy Hypothesis $(1)$ of Definition \ref{def:MR}, but satisfy the other Hypothesise of Definition \ref{def:MR}. For the sake of the length of the paper, we omit the detailed proof.

\begin{example}
Consider the intervals
\[
I_1 \coloneqq \left(\frac{1}{3}, \frac{2}{3}\right),\ \ I_2 \coloneqq \left(\frac{1}{9}, \frac{2}{9}\right),\ \ I_3 \coloneqq \left(\frac{7}{9}, \frac{8}{9}\right),\ \ I_4 \coloneqq \left(\frac{1}{27}, \frac{2}{27}\right),\ \ I_5 \coloneqq \left(\frac{7}{27}, \frac{8}{27}\right),
\]
\[
I_6 \coloneqq \left(\frac{19}{27}, \frac{20}{27}\right),\ \ I_7 \coloneqq \left(\frac{25}{27}, \frac{26}{27}\right),\ \ \cdots.
\]
These are the open intervals removed at each step in the construction of the standard middle-third Cantor set. Notably, they are not arranged in a left-to-right or right-to-left order within the unit interval. For any $2^{k-1} \leq n <2^k$ with $k \in \mathbb N$, each open interval $I_n$ is of the form
\[
\left(\frac{3j+1}{3^n}, \frac{3j+2}{3^n}\right),
\]
where $j$ is an integer between $0$ and $3^{n-1}-1$ whose base-$3$ representation contains only the digits $0$ and $2$. Define $T$ to be a linear map on each $I_n$, sending the left endpoint of $I_n$ to $0$ and the right endpoint to $1$. In this setting, the parameter $\gamma$ in Hypothesis $(5)$ of Definition \ref{def:MR} is given by $\gamma = \frac{\log 3}{\log 2}$. Regarding the fast Lyapunov spectrum, for any function $\psi:\mathbb{N}\to\mathbb{R}_{>0}$ be a function satisfying $\psi(n)/n \to \infty$ as $n\to \infty$, we obtain the lower bound
\[
F_{\psi}(\infty) \geq \frac{\log 2}{\log 3}.
\]
In particular, choose a function $\psi(n)=2^{2^{2^{2^n}}}$. Since $B_\psi = \infty$ in this case, the conclusion of Theorem \ref{t.Fast_Lyapunov} would imply that $F_{\psi}(\infty)=0$, which is incorrect.
\end{example}
\begin{example}
Consider the intervals
\[
I_1 \coloneqq \left(\frac{2}{3}, 1\right),\ \ I_2 \coloneqq \left(\frac{8}{27}, \frac{1}{3}\right),\ \ I_3 \coloneqq \left(\frac{2}{9}, \frac{7}{27}\right), \ I_4 \coloneqq \left(\frac{26}{243}, \frac{1}{9}\right),\ \ I_5 \coloneqq \left(\frac{8}{81},\frac{25}{243}\right),
\]
\[
I_6 \coloneqq \left(\frac{20}{243}, \frac{7}{81}\right),\ \ I_7 \coloneqq \left(\frac{2}{27}, \frac{19}{243}\right),\ \ \cdots.
\]
These open intervals are arranged in a right-to-left order within the unit interval, but they do not share endpoints. Observe that, for any $2^{k-1} \leq n <2^k$ with $k \in \mathbb N$, each open interval $I_n$ has diameter $3^{-(2k-1)}$. Define $T$ to be a linear map on each $I_n$, sending the left endpoint of $I_n$ to $0$ and the right endpoint to $1$. In this setting, the parameter $\gamma$ in Hypothesis $(5)$ of Definition \ref{def:MR} is given by $\gamma = \frac{2\log 3}{\log 2}$. Regarding the fast Lyapunov spectrum, for the function $\psi(n)=2^n$, we obtain the lower bound
\[
\overline{F}_{\psi}(\alpha) \geq \frac{\log 2}{3\log 3}
\quad \text{and} \quad
\underline{F}_{\psi}(\alpha)\geq \frac{\log 2}{3\log 3}.
\]
However, since $b_\psi = B_\psi = 2$ in this case, the conclusion of Theorem \ref{t.Upper_and_lower_fast_Lyapunov_spectrum} would imply that
\[
\overline{F}_{\psi}(\alpha)=
\underline{F}_{\psi}(\alpha) = \frac{\log 2}{4\log 3-\log 2} < \frac{\log 2}{3\log 3},
\]
which is incorrect.
\end{example}

\medskip

The paper is organized as follows. Section \ref{s.distributiondigits} is devoted to the study of symbolic codings and their distribution for Markov-R\'{e}nyi maps. In Section \ref{s.proofoftheorema}, we provide the proof of Theorem \ref{t.Fast_Lyapunov}. In Section \ref{s.proofoftheoremb}, we present the proof of Theorem \ref{t.Upper_and_lower_fast_Lyapunov_spectrum}. Section \ref{s.appendix} contains an appendix in which we construct a non-decreasing function $g_{\psi}$ with stronger properties than the one in Subsection \ref{constrction}. We believe this construction has an independent interest on its own.

\section{Symbolic codings and their distribution}\label{s.distributiondigits}
In this section, we review the symbolic codings and their distribution for Markov-R\'{e}nyi maps.
We follow the convention $\mathbb{N}=\{1,2,3,\cdots\}$ the set of positive integers, $\mathbb{R}$ the set of real numbers, $\mathbb{R}_{>0}\coloneqq\{x\in \mathbb{R}:x>0\}$ and $\mathbb{R}_{\geq 0}\coloneqq\mathbb{R}_{>0}\cup\{0\}$.
For a subset $A\subseteq \mathbb{R}$, denote by $|A|\coloneqq \sup\{|x-y|:x,y\in A\}$ the diameter of $A$, $A^n \coloneqq \{(x_1,\dots,x_n): x_i \in A, \forall 1\leq i\leq n\}$ the Cartesian product of $A$, $\# A$ the cardinality of $A$, $\overline{A}$ the closure of A, and $\partial A$ the boundary of $A$. For a function $f:A \to \mathbb{R}$, and $B \subseteq A$, let $f|_B$ be the restriction of $f$ on $B$.

\subsection{Symbolic codings}
The Markov structure of the Markov-R\'{e}nyi map implies that it is semi-conjugate to the full shift on a countable alphabet (see, e.g., \cite{PW,Sar09}). By Hypothesis (1) of Definition \ref{def:MR}, it follows that
\[
I^* \coloneqq \bigcup_{n \in \mathbb N} \overline{I}_n
\]
is a subinterval of $[0,1]$. For a Markov-R\'{e}nyi map $T$, let $$\mathcal{Q}\coloneqq \bigcup^\infty_{k=0}T^{-k}\left(\{0,1\} \cup \bigcup_{n \in \mathbb N}\partial I_n\right).$$
Then $\mathcal{Q}$ is a countable subset of $I^*$.
Define $$\Lambda\coloneqq I^*\setminus \mathcal Q$$
as the \emph{repeller} of $T$.
Thus, every $x\in \Lambda$ has a unique \emph{symbolic coding} $(a_1(x),a_2(x),a_3(x),\dots) \in \mathbb N^{\mathbb N}$ since for any $n\in\mathbb N$, there exists a unique interval $I_{a_n(x)}$ such that
\begin{equation}\label{coding}
 T^{n-1}(x) \in I_{a_n(x)}.
\end{equation}

The following result establishes a relationship between $\log|(T^n)'(x)|$ and $\sum_{k=1}^n\log a_k(x)$. It allows us to transform the study of fast Lyapunov exponents to the investigation of the fast growth rate of $\sum_{k=1}^n\log a_k(x)$.

\begin{proposition}\label{l.Transformation}
               Given a Markov-R\'{e}nyi map $T$, for any $x\in \Lambda$ and $n \in \mathbb N$, we have
        \begin{equation*}\label{e.O(n)}
               \left|\gamma\sum_{k=1}^n\log a_k(x)-\log|(T^n)'(x)|\right|
                \leq n\log C,
        \end{equation*}
        where the parameters $\gamma$ and $C$ are given in Hypothesis $(5)$ of Definition \ref{def:MR}.
\end{proposition}

\begin{proof}
Given a Markov-R\'{e}nyi map $T$, for any $x\in  \Lambda$ and $n\in \mathbb N$, it follows from the chain rule that
        \begin{equation}\label{e.log_T_n_sum_log_T_k}
                \log|(T^n)'(x)|=
                \sum_{k=1}^n\log
                |T'(T^{k-1}(x))|.
        \end{equation}
For each $k \in \mathbb N$, the integer $a_k(x)$ is uniquely determined by the relation $T^{k-1}(x)\in I_{a_k(x)}$. From Hypothesis (5) in Definition \ref{def:MR}, we derive that
 \begin{equation*}\label{e.estimation_a_k_R^k-1}
                C^{-1}a_k^{\gamma}(x)\leq |T'(T^{k-1}(x))|\leq Ca_k^{\gamma}(x).
 \end{equation*}
That is,
        \begin{equation*}\label{e.estimation_log_a_k_R^k-1}
                \gamma\log a_{k}(x)-\log C\leq \log |T'(T^{k-1}(x))|\leq \gamma\log a_{k}(x)+\log C.
        \end{equation*}
Combining this with \eqref{e.log_T_n_sum_log_T_k}, we conclude that
\begin{equation*}
                \left|\gamma\sum_{k=1}^n\log
                a_k(x)-\log|(T^n)'(x)|\right|
                \leq n\log C.
        \end{equation*}
The proof is completed.
\end{proof}

\subsection{Cylinder sets}

Let $ T$ be a Markov-R\'{e}nyi map. For any $n\in \mathbb N$ and $(i_1,\dots, i_n)\in \mathbb N^n$, denote by
\begin{equation}\label{def:Cylinder}
I_n(i_1,\dots,i_n) \coloneqq\bigcap_{j=1}^{n} T^{-j+1} I_{i_j}
\end{equation}
the \emph{cylinder set} of order $n$ associated with $(i_1,\dots, i_n)$.

Except for a countable set, the cylinder set $I_n(i_1,\dots,i_n)$ coincides with the set of points whose symbolic codings begin with $i_1,\dots, i_n$.
Cylinder sets and their diameters play a fundamental role in the study of the Hausdorff dimension of sets associated with Markov-R\'{e}nyi maps.
The following estimate can be derived from Hypothesis (5) of Definition \ref{def:MR} of Markov-R\'{e}nyi maps.

\begin{proposition}\label{p.cylinder}
For any $n\in \mathbb N$ and $(i_1,\dots, i_n)\in \mathbb N^n$, the cylinder set $I_n(i_1,\dots,i_n)$ is an interval, and its diameter satisfies
\begin{equation*}\label{e.cylinder}
   \frac{C^{-n}}{(i_1 \cdots i_{n})^{\gamma}}
   \leq |I_n(i_1,\dots,i_n)|
   \leq \frac{C^{n}}{(i_1 \cdots i_{n})^{\gamma}},
\end{equation*}
where the parameters $\gamma$ and $C$ are given in Hypothesis (5) of Definition \ref{def:MR}.
\end{proposition}

\begin{proof}
Fix $n\in \mathbb N$ and $(i_1,\dots, i_n)\in \mathbb N^n$. We first show that the cylinder set is an interval. The proof follows from the arguments provided by Sarig \cite[page 13]{Sar09}. Note that $T_{i_1}$ is $C^1$ monotonic, and one-to-one. By Hypothesis (4) of Definition \ref{def:MR}, we see that $T(I_{i_1}) \supset I_{i_2}$, and so
$$T(I_{i_1}\cap T^{-1}I_{i_2})=T(I_{i_1})\cap I_{i_2} =I_{i_2}.$$ This implies that $I_{i_1}\cap T^{-1}I_{i_2}$ is an interval which is mapped by $T$ onto $I_{i_2}$.
Similarly, note that $T(I_{i_2}) \supset I_{i_3}$ and $T^2(I_{i_1}\cap T^{-1}I_{i_2}) =T(I_{i_2})$. We derive that the restriction of $T^2$ to the interval $I_{i_1}\cap T^{-1}I_{i_2}$ is equal to the restriction of $T_{i_2} \circ T_{i_1}$ to that interval, which is $C^1$ monotonic, and one-to-one. It follows that $I_{i_1}\cap T^{-1}I_{i_2}\cap T^{-2}I_{i_3}$ is an interval which is mapped by $T^2$ onto $I_{i_3}$. Continuing in this way, we see that $I_n(i_1,\dots,i_n)$ defined in \eqref{def:Cylinder} is an interval which is mapped by $T^{n-1}$ onto $I_{i_n}$.

The estimate of the diameter of $I_n(i_1,\dots,i_n)$ is a consequence of Hypothesis (5) of Definition \ref{def:MR} and the mean value theorem.
\end{proof}

\subsection{The distribution of codings}

Let $\{s_n\}_{n\geq 1}$ and $\{t_n\}_{n\geq 1}$ be sequences of real numbers such that $s_n, t_n\geq 2$ for all $n\geq 1$.
Define
\begin{equation}\label{e.def_E_s_n_t_n}
E(\{s_n\},\{t_n\})  \coloneqq
        \left\{x\in  \Lambda:s_n<a_n(x)\leq s_n+t_n,
        \forall n\geq 1\right\}.
\end{equation}

Throughout this section, we assume that the following two hypotheses hold:
\begin{equation}\label{e.sumlog_s_k/n=infty}
        \lim_{n\to \infty}
        \frac{\sum_{k=1}^n\log s_k}
        {n}=\infty
\end{equation}
and
\begin{equation}\label{e.inf_s_n/t_n>c>0}
        \liminf_{n\to \infty}\frac{s_n}{t_n}>0.
\end{equation}
The Hausdorff dimension of $E(\{s_n\},\{t_n\})$ is as follows.

\begin{proposition}\label{p.Hausdorff_dimension_E_s_n_t_n}
        Assume that $\{s_n\}_{n\geq 1}$ and $\{t_n\}_{n \geq 1}$ satisfy Hypotheses \eqref{e.sumlog_s_k/n=infty} and \eqref{e.inf_s_n/t_n>c>0}. Then
        \begin{equation}\label{e.Prop_dimension}
                \dim_{\mathrm{H}}
                E (\{s_n\},\{t_n\})
                =\liminf _{n\to \infty}
                \frac{\sum_{k=1}^n\log t_k}
                {\gamma\sum_{k=1}^{n+1}\log s_k-\log t_{n+1}}.
        \end{equation}
\end{proposition}

Before we proceed the proof, let us first provide two useful lemmas from Falconer \cite{Fal}.

\begin{lemma}[{\cite[Example 4.6]{Fal}}]\label{l.lowerbound}
        Let $E_0\supset E _1\supset\cdots $ be a decreasing sequence of sets and $E=\bigcap_{n\geq 0}E_n$. Assume that each $E_n$ is a union of a finite number of disjoint closed intervals (called basic intervals of order $n$) and each basic interval in $E_{n-1}$ contains $m_n$ intervals of $E_n$ that are separated by gaps of at least $\epsilon_n$. If $m_n\geq 2$ and $\epsilon_{n}>\epsilon_{n+1}>0$ for every $n\geq 1$, then
        \begin{equation}\label{e.lemma_lower_bound}
                \dim_{\mathrm{H}}E\geq
                \liminf_{n\to \infty}
                \frac{\log(m_1m_2\cdots m_{n-1})}
                {-\log(m_n\epsilon_n)}.
        \end{equation}
\end{lemma}
\begin{lemma}[{\cite[Proposition 4.1]{Fal}}]\label{l.upperbound}
        Suppose $F$ can be
covered by $\mathcal{N}_n$ sets of diameter at most $\delta_n$
        with $\delta_n\rightarrow 0$
        as $n\rightarrow \infty$.
        Then
        \begin{equation*}
                \dim_{\mathrm{H}}F
                \leq \liminf_{n\to \infty}
                \frac{\log\mathcal{N}_n}{-\log\delta_n}.
        \end{equation*}
\end{lemma}

We are now ready to prove Proposition \ref{p.Hausdorff_dimension_E_s_n_t_n}.
\begin{proof}[Proof of Proposition \ref{p.Hausdorff_dimension_E_s_n_t_n}]
We will apply Lemmas \ref{l.lowerbound} and \ref{l.upperbound} to prove Proposition \ref{p.Hausdorff_dimension_E_s_n_t_n}. To do this, we first characterize the set $E(\{s_n\},\{t_n\})$.
For any $n \in \mathbb N$, let
\begin{equation}\label{e.def_C_n}
                \mathcal{C} _{n}\coloneqq \{
                (i_1, \dots,i_{n})\in
                \mathbb{N}^n:s_j<i_{j}\leq s_j+t_j,
               1\leq j\leq n
                \}.
        \end{equation}
For any $(i_1, \dots, i_{n})\in \mathcal{C}_{n}$, we define the \textit{basic interval} of order $n$ as follows:
        \begin{equation}\label{e.def_J_n}
                J_n(i_1, \dots, i_{n})
                \coloneqq \bigcup_
                {s_{n+1}<k\leq s_{n+1}+t_{n+1}} \overline{I}_{n+1}(i_1, \dots, i_{n}, k).
        \end{equation}
Note that $J_n(i_1, \dots, i_{n})$ is a closed interval. Define $E_0\coloneqq [0,1]$, and for any $n \in \mathbb N$, let
        \begin{equation}\label{e.def_F_n}
                E_n\coloneqq
                \bigcup_{(i_1, \dots, i_{n})\in
                        \mathcal{C} _{n}}
                J_n(i_1, \dots, i_{n}).
        \end{equation}
Then
\begin{equation*}\label{equ:esntn}
E(\{s_n\},\{t_n\}) \subseteq \cap_{n\geq 0}E_n \subseteq E(\{s_n\},\{t_n\}) \cup \mathcal{Q}.
\end{equation*}
Note that $\mathcal{Q}$ is a countable set. These lead us to the conclusion that $E(\{s_n\},\{t_n\})$ and $\cap_{n\geq 0}E_n$ have the same Hausdorff dimension, namely
\begin{equation}\label{equ:dengshi}
\dim_{\mathrm{H}} E(\{s_n\},\{t_n\}) = \dim_{\mathrm{H}}(\cap_{n\geq 0}E_n).
\end{equation}

We will proceed the proof in the following two parts.

\textbf{Part I:} In this part, we use Lemma \ref{l.lowerbound} to establish a lower bound for $\dim_{\mathrm{H}} E(\{s_n\},\{t_n\})$. Based on \eqref{equ:dengshi}, we do the lower bound by identifying $E_n$ in Lemma \ref{l.lowerbound} with the set in Equation \eqref{e.def_F_n}.
It is important to note that $J_n(i_1, \dots, i_{n})$ serves as a refinement of $I_n(i_1, \dots, i_{n})$, such that different basic intervals maintain a non-trivial gap between them. The estimation of $m_n$ follows directly from \eqref{e.def_J_n}. Applying Proposition \ref{p.cylinder}, we are able to estimate $\epsilon_n$. The combination of these estimations result in a lower bound for $\dim_{\mathrm{H}} E(\{s_n\},\{t_n\})$, as indicated in \eqref{e.lemma_lower_bound}.

First, we establish an estimate for $m_n$. By the structure of basic intervals in \eqref{e.def_J_n}, we deduce that each basic interval of order $n-1$ contains
        \begin{equation}\label{e.estimation_m_n}
                \frac{t_n}{2}<\lfloor t_n\rfloor \leq
                m_n=\lfloor s_n+t_n\rfloor
                -\lfloor s_n\rfloor <t_n+1\leq 2t_n
        \end{equation}
        basic intervals of order $n$.

Next let us estimate $\epsilon_n$. Assume that $J_n(i_1, \dots, i_{n})$ and $J_n(i^*_1, \dots, i^*_{n})$ are two different basic intervals in $E_n$. Then they are separated by
        \begin{equation*}
                I_{n+1}(i_1, \dots, i_{n}, 1)\quad
                \mathrm{or} \quad
                I_{n+1}(i^*_1, \dots, i^*_{n}, 1).
        \end{equation*}
This means that the gap between these two basic intervals is at least
\begin{equation*}
                \epsilon_n\coloneqq \min_{(i_1, \dots, i_{n})\in
                        \mathcal{C} _{n}} \left\{|I_{n+1}(i_1, \dots, i_{n}, 1)|\right\}.
        \end{equation*}
In view of \eqref{e.inf_s_n/t_n>c>0}, there exists a constant $\theta >0$ such that $s_j > \theta t_j$ for all $j \in \mathbb N$. Hence, for any $(i_1, \dots, i_{n})\in\mathcal{C} _{n}$, we have $i_{j}\leq s_j+t_j < (1+1/\theta)s_j$ for all $1\leq j\leq n$. It follows from Proposition \ref{p.cylinder} that
 \begin{alignat*}{2}
                |I_{n+1}(i_1, \dots, i_{n}, 1)|
                \geq  \frac{C^{-n-1}}{((1+1/\theta)s_1 (1+1/\theta)s_2\cdots(1+1/\theta)s_n)^{\gamma}}\notag
                = \frac{C^{-n-1}}{(1+1/\theta)^{n\gamma}(s_1s_2\cdots s_n)^{\gamma}},
        \end{alignat*}
which implies that
 \begin{alignat}{2} \label{a.estimate_epsilon_n}
                \epsilon_n \geq \frac{C^{-n-1}}{(1+1/\theta)^{n\gamma}(s_1s_2\cdots s_n)^{\gamma}}.
        \end{alignat}
Combining Lemma \ref{l.lowerbound} with \eqref{equ:dengshi}, \eqref{e.estimation_m_n} and \eqref{a.estimate_epsilon_n}, we conclude that
  \begin{alignat*}{2}
                \dim_{\mathrm{H}}E(\{s_n\},\{t_n\})\notag
                \geq&  \liminf_{n\to \infty}
                \frac{\log (m_1 m_2 \cdots m_n)}
                {-\log (m_{n+1}\epsilon_{n+1})}\\
                \geq&  \liminf_{n\to \infty}
                \frac{\log (t_1t_2\cdots t_n)-n\log 2}
                {\gamma \log(s_1s_2\cdots s_ns_{n+1})-\log t_{n+1}
                +\ell_1(n)} \\
                =& \liminf_{n\to \infty}
                \frac{\log (t_1t_2\cdots t_n)}
                {\gamma\log(s_1s_2\cdots s_ns_{n+1})-\log t_{n+1}},
                && \  \big(\text{by}\ \eqref{e.sumlog_s_k/n=infty}\big)
        \end{alignat*}
        where $\ell_1(n)\coloneqq \left(\gamma \log(1+1/\theta)+\log C\right)n+\log\left(2C^2(1+1/\theta)^{\gamma}\right)$ is a linear function of $n$.
Then we obtain the desired lower bound for $\dim_{\mathrm{H}}E(\{s_n\},\{t_n\})$ as \eqref{e.Prop_dimension}.

\textbf{Part II:} In this part, we use Lemma \ref{l.upperbound} to give an upper bound for $\dim_{\mathrm{H}} E(\{s_n\},\{t_n\})$ by covering $E(\{s_n\},\{t_n\})$ with $E_n$ for each $n\geq 1$ with $\mathcal{N} _n \coloneqq \#\mathcal{C} _n$, the number of basic intervals of order $n$.
The estimation of $\mathcal{N} _n$ follows directly from the definition of $\mathcal{C} _n$ in \eqref{e.def_C_n} by
        \begin{equation*}\label{e.estimate_N_n}
                \mathcal{N} _n<2t_1\cdot2t_2\cdots 2t_n=
                2^nt_1t_2\cdots t_n.
        \end{equation*}
We estimate the diameter of the basic interval of order $n$. For any $(i_1, \dots, i_{n})\in\mathcal{C} _{n}$, we have $i_{j} >s_j$ for all $1\leq j\leq n$. Note that
 \begin{equation*}
 \sum_{s_{n+1}<k\leq s_{n+1}+t_{n+1}}\frac{1}{k^{\gamma}} <\frac{t_{n+1}+1}{s_{n+1}^{\gamma}} \leq  \frac{2t_{n+1}}{s_{n+1}^{\gamma}}.
        \end{equation*}
By \eqref{e.def_J_n} and Proposition \ref{p.cylinder}, we deduce that
    \begin{alignat*}{2}
                |J_n(i_1, \dots, i_{n})|
                \leq \frac{C^{n+1}}{(i_1\cdots i_{n})^{\gamma}}\sum_{s_{n+1}<k\leq s_{n+1}+t_{n+1}}
                \frac{1}{k^{\gamma}}
               < \frac{2C^{n+1}t_{n+1}}{(s_1\cdots
                s_ns_{n+1})^{\gamma}}.
        \end{alignat*}
For any $n \in \mathbb N$, let
  \begin{equation*}
 \delta_n \coloneqq  \frac{2C^{n+1}t_{n+1}}{(s_1\cdots
                s_ns_{n+1})^{\gamma}}.
        \end{equation*}
Hence, $E(\{s_n\},\{t_n\})$ is covered by $\mathcal{N}_n$ basic intervals of diameter at most $\delta_n$. According to \eqref{equ:dengshi} and Lemma \ref{l.upperbound}, we derive that
 \begin{alignat*}{2}
                \dim_{\mathrm{H}}E(\{s_n\},\{t_n\})
               & \leq \liminf_{n\to \infty}\frac{\log \mathcal{N} _n}
                {-\log \delta_n}\\
               & \leq  \liminf_{n\to \infty}
                \frac{n\log 2+\log (t_1t_2\cdots t_n)}
                {\gamma \log(s_1s_2\cdots s_ns_{n+1})-\log t_{n+1}
                +\ell_2(n)} \\
               & \leq  \liminf_{n\to \infty}
                \frac{\log (t_1t_2\cdots t_n)}
                {\gamma\log(s_1s_2\cdots s_ns_{n+1})-\log t_{n+1}},
                && \quad  \big( \text{by}\
                \eqref{e.sumlog_s_k/n=infty}\big) \notag
        \end{alignat*}
where $\ell_2(n)\coloneqq -(n+1)\log C-\log 2$ is a linear function of $n$. Then we obtain the required upper bound for $\dim_{\mathrm{H}}E(\{s_n\},\{t_n\})$ as \eqref{e.Prop_dimension}. Therefore, the proof of Proposition \ref{p.Hausdorff_dimension_E_s_n_t_n} is completed.
\end{proof}

Recall that
            \begin{equation}\label{e.def_D}
        E(\{e^n\}, \{e^n\})\coloneqq
         \left\{x\in  \Lambda:e^n<a_n(x)\leq 2e^{n},
        \forall n\geq 1\right\}.
        \end{equation}
        As a consequence of Proposition \ref{p.Hausdorff_dimension_E_s_n_t_n}, we are able to determine its Hausdorff dimension.

        \begin{cor}\label{c.D}
        \begin{equation*}
        \dim_{\mathrm{H}}E(\{e^n\},\{e^n\}) =\frac{1}{\gamma}.
        \end{equation*}
        \end{cor}

We end this section by providing the Hausdorff dimension of the following set
        \begin{equation}\label{e.def_pi_infty}
                \Lambda_{\infty}\coloneqq
                \left\{x\in \Lambda:
                \limsup_{n\to\infty}\frac{\sum^n_{k=1}\log a_k(x)}{n}
                =\infty\right\}.
        \end{equation}

        \begin{proposition}\label{l.Hausdorff_dimension_limsup_log_pi/n}
        \begin{equation*}
                \dim_{\mathrm{H}}\Lambda_{\infty}
                =\frac{1}{\gamma}.
        \end{equation*}
\end{proposition}

\begin{proof}
For the lower bound of $\dim_{\mathrm{H}}\Lambda_{\infty}$, since $E(\{e^n\},\{e^n\})\subset \Lambda_{\infty}$, it follows from Corollary \ref{c.D} that
        \begin{equation*}
        \dim_{\mathrm{H}}\Lambda_{\infty}\geq
        \dim_{\mathrm{H}}
        E(\{e^n\},\{e^n\})
        =\frac{1}{\gamma}.
        \end{equation*}

 For the upper bound of $\dim_{\mathrm{H}}\Lambda_{\infty}$, let $\epsilon >0$ be an arbitrary real number and $s\coloneqq (1+2\epsilon)/\gamma$.
Choosing a sufficiently large number $K>1$ such that
        \begin{equation}\label{e.def_K_K>4}
                K^{\epsilon}>2C^sJ_{\epsilon}
                \quad \mathrm{and}\quad
                K^{\gamma}>2C,
        \end{equation}
         where
        \begin{equation}\label{e.def_J_epsilon}
             J_{\epsilon} \coloneqq \sum_{n=1}^{\infty}
        \frac{1}{n^{1+\epsilon}}<\infty.
        \end{equation}
For any $n \in \mathbb N$, define
 \begin{equation}\label{e.def_C_n_K}
             \mathcal{C}_n(K)
        \coloneqq\{
                (i_1,\dots,i_{n})\in
                \mathbb{N}^n:
                i_1\cdots i_{n}\geq K^n
        \}.
        \end{equation}
Then, for any $N\in\mathbb{N}$, the set $\Lambda_{\infty}$ is covered by
 \begin{equation}\label{e.cover_Pi_infty}
                \bigcup_{n=N}^{\infty}
                \bigcup_{(i_1,\dots,i_{n})\in \mathcal{C}_n(K)}
                I_n(i_1,\dots,i_{n}).
        \end{equation}
For any $\delta >0$, by Proposition \ref{p.cylinder},  \eqref{e.def_K_K>4} and \eqref{e.def_C_n_K}, there exists $M \coloneqq \lfloor \log _{2} 1/\delta\rfloor +1$,
        such that for all $n>M$ and $(i_1,\dots,i_{n})\in \mathcal{C}_n(K)$,
        \begin{equation*}\label{e.delta_cover_condition}
             |I_n(i_1,\dots, i_{n})|
            \leq \frac{C^{n}}
            {(i_1\cdots i_{n})^{\gamma}}
            \leq \frac{C^{n}}{K^{\gamma n}}
            \leq \left(\frac{1}{2}\right)^n
            \leq \left(\frac{1}{2}\right)^M
            <\delta.
        \end{equation*}
 Consequently, when $N>M$, \eqref{e.cover_Pi_infty} is a $\delta$-cover of $\Lambda_{\infty}$. We are ready to calculate the Hausdorff measure of $\Lambda_{\infty}$. For any  $(i_1,\dots,i_{n})\in \mathcal{C}_n(K)$, we deduce from Proposition \ref{p.cylinder} that
 \[
  |I_n(i_1\cdots i_{n})|^s \leq \frac{C^{sn}}{(i_1\cdots i_{n})^{1+2\epsilon}} \leq \left(\frac{C^{s}}{K^\epsilon}\right)^n \frac{1}{(i_1\cdots i_{n})^{1+\epsilon}}.
 \]
Therefore, we conclude that
      \begin{alignat*}{2}
                \mathcal{H}_{\delta}^{s}(\Lambda_{\infty})
                \leq &\liminf_{N\to \infty}
                \sum_{n=N}^{\infty}\sum_{(i_1,\dots, i_{n})\in \mathcal{C}_{n}(K)}
                |I_n(i_1\cdots i_{n})|^s\\
                \leq& \liminf_{N\to \infty}
                \sum_{n=N}^{\infty}\left(\frac{C^{s}}{K^\varepsilon}\right)^n
                \sum_{(i_1,\dots, i_{n})\in \mathcal{C}_{n}(K)}
                \frac{1}{(i_1\cdots i_{n})^{1+\epsilon}}\\
                \leq& \liminf_{N\to \infty}\sum_{n=N}^{\infty}
                \left(\frac{C^sJ_{\epsilon}}{K^{\epsilon}}\right)^{n}\\
                \leq& \liminf_{N\to \infty}\sum_{n=N}^{\infty}
                \frac{1}{2^n}=0.
        \end{alignat*}
Letting $\delta \to 0^{+}$, we have $\mathcal{H}^{s}(\Lambda_{\infty})=0$, and so $\dim_{\text{H}}\Lambda_{\infty}\leq s$. Since $\varepsilon$ is arbitrary, we obtain $$\dim_{\text{H}} \Lambda_{\infty} \leq \frac{1}{\gamma}.$$
\end{proof}

Note that $E(\{e^n\},\{e^n\}) \subset L(\infty) \subset \Lambda_\infty$. We deduce from Corollary \ref{c.D} and Propositions \ref{l.Transformation} and \ref{l.Hausdorff_dimension_limsup_log_pi/n} that
\[
\dim_{\rm H}L(\infty) = \frac{1}{\gamma}.
\]
This shows that the Lyapunov spectrum $L(\alpha)$ is continuous at infinity. Hence, the proof of Proposition \ref{LSinfty} is completed.

\section{Proof of Theorem \ref{t.Fast_Lyapunov}}\label{s.proofoftheorema}

In this section, we provide the proof of Theorem \ref{t.Fast_Lyapunov}.
Assume that $\psi:\mathbb{N}\to\mathbb{R}_{>0}$ is a function such that $\psi(n)/n \to \infty$ as $n\to \infty$.

\subsection{Case $\alpha =0$}

For any $x\in \Lambda$ and $n \in \mathbb N$, let
\begin{equation}\label{Pin(x)}
 \Pi_{n}(x) \coloneqq\prod^n_{i=1}a_i(x)
\end{equation}
be the product of the first $n$ terms in the coding sequence of $x$. The behavior of $\Pi_{n}(x)$ is as follows.

\begin{lemma}\label{lem:Acase0}
For Lebesgue almost every $x\in I^*$,
\begin{equation}\label{ineq:sumak}
\limsup_{n \to \infty} \frac{1}{n}\log  \Pi_{n}(x)<\infty.
\end{equation}
\end{lemma}

\begin{proof}
For any $k,n \in \mathbb N$, we claim that
\begin{equation}\label{ineq:Pn}
\mathcal{L}\left\{x\in \Lambda: 2^{kn} < \Pi_{n}(x) \leq 2^{(k+1)n}\right\} < \left(\frac{2e C(k+2)}{2^{k(\gamma-1)}}\right)^n,
\end{equation}
where $\mathcal{L}$ denotes the Lebesgue measure on $[0,1]$. To see this, let
$$
\mathcal{B}_n(k):= \left\{(\sigma_1,\dots,\sigma_n)\in \mathbb N^n: 2^{kn} <\sigma_1\cdots\sigma_n \leq 2^{(k+1)n}\right\}.
$$
Then the set in \eqref{ineq:Pn} can be written as
$$
\bigcup_{(\sigma_1,\dots,\sigma_n) \in \mathcal{B}_n(k)} I_n(\sigma_1,\dots,\sigma_n).
$$
For any $(\sigma_1,\dots,\sigma_n) \in \mathcal{B}_n(k)$, we deduce from Proposition \ref{p.cylinder} that
\begin{equation*}
\mathcal{L}(I_n(\sigma_1,\dots,\sigma_n)) = |I_n(\sigma_1,\dots,\sigma_n)| \leq \frac{C^n}{(\sigma_1\cdots\sigma_n)^\gamma} < \frac{C^n}{2^{k\gamma n}}.
\end{equation*}

Next, we estimate the cardinality of $\mathcal{B}_n(k)$. For any $(\sigma_1,\dots,\sigma_n) \in \mathcal{B}_n(k)$, let $\tau_i := \lfloor\log_2 \sigma_i\rfloor$ for all $1\leq i \leq n$. Then $\tau_i \in \mathbb N_{\geq 0}$ and $2^{\tau_i} \leq \sigma_i <2^{\tau_i+1}$. This implies that each value of $\tau_i$ corresponds to $2^{\tau_i}$ many possible values of $\sigma_i$. Since $2^{kn} <\sigma_1\cdots\sigma_n \leq 2^{(k+1)n}$, it follows that $(k-1)n < \sum^n_{i=1}\tau_i \leq (k+1)n$. Thus,
\begin{align}\label{Bsolution}
\#\mathcal{B}_n(k) \leq 2^{(k+1)n} \cdot \# \left\{(\tau_1,\dots,\tau_n)\in \mathbb N^n_{\geq 0}: \sum^n_{i=1}\tau_i \leq (k+1)n\right\}.
\end{align}
The second term on the right-hand side of \eqref{Bsolution} counts the number of nonnegative integer solutions to the inequality $\sum^n_{i=1}\tau_i \leq (k+1)n$. Equivalently, it represents the number of ways to distribute at most $(k+1)n$ indistinguishable units among $n$ nonnegative integer variables. This quantity is given by the well-known combinatorial identity:
$$
\sum^{(k+1)n}_{j=0} \binom{j+n-1}{n-1} = \binom{(k+2)n}{n}.
$$
Since $n! > n^n e^{-n}$ for all $n \in \mathbb N$, we deduce that
$$
\binom{(k+2)n}{n} < \frac{\left((k+2)n\right)^n}{n!}< \left(e(k+2)\right)^n.
$$
Note that $2^{\sum^n_{i=1}\tau_i} \leq 2^{(k+1)n}$. Combining these with \eqref{Bsolution}, we see that
\[
\#\mathcal{B}_n(k) \leq \left(e(k+2)\cdot2^{(k+1)}\right)^n.
\]
Therefore,
\begin{align*}
\mathcal{L}\left\{x\in \Lambda: 2^{kn}< \Pi_{n}(x) \leq 2^{(k+1)n}\right\} &\leq \# \mathcal{B}_n(k)  \cdot \mathcal{L}(I_n(\sigma_1,\dots,\sigma_n))\\
&< \left(e(k+2)\cdot2^{(k+1)}\right)^n \cdot \frac{C^n}{2^{k\gamma n}}= \left(\frac{2e C(k+2)}{2^{k(\gamma-1)}}\right)^n.
\end{align*}

Finally, we turn to the proof of \eqref{ineq:sumak}. To this end, we choose a sufficiently large integer $k_0\geq 2$ such that for all $k \geq k_0$,
$$
\frac{2e C(k+2)}{2^{k(\gamma-1)}} \leq \frac{1}{2^{k(\gamma-1)/2}}.
$$
For any $n \in \mathbb N$, it follows from \eqref{ineq:Pn} that
\begin{align*}
\mathcal{L}\left\{x\in \Lambda:  \Pi_{n}(x) >2^{k_0n}\right\} &= \sum^\infty_{k=k_0} \mathcal{L}\left\{x\in \Lambda: 2^{kn}< \Pi_{n}(x) \leq 2^{(k+1)n}\right\}\\
&\leq \sum^\infty_{k=k_0} \left(\frac{1}{2^{n(\gamma-1)/2}}\right)^k \leq \frac{2^{(\gamma-1)/2}}{2^{(\gamma-1)/2}-1}\cdot \frac{1}{2^{(\gamma-1)n}}
\end{align*}
which yields that
$$
\sum^\infty_{n=1} \mathcal{L}\left\{x\in \Lambda:  \Pi_{n}(x) >2^{k_0n}\right\} <\infty.
$$
From the Borel-Cantelli lemma, we conclude that for Lebesgue almost every $x\in I^*$, $\Pi_{n}(x) \leq 2^{k_0n}$ for sufficiently large $n \in \mathbb N$. Thus, \eqref{ineq:sumak} holds as required.
\end{proof}

The proof of Theorem \ref{t.Fast_Lyapunov} for the case $\alpha =0$ follows directly from Proposition \ref{l.Transformation} and Lemma \ref{lem:Acase0}.

\subsection{Case $0<\alpha< \infty$}
For any $0<\alpha< \infty$, define
\begin{equation}\label{e.def_E_psi_alpha}
        E_{\psi}(\alpha)\coloneqq
        \left\{x\in \Lambda:\lim_{n \to \infty}
        \frac{\sum_{k=1}^n \log a_{k}(x)}{\psi(n)}= \alpha \right\}.
\end{equation}
By Proposition \ref{l.Transformation}, we see that $J_\psi(\alpha) =E_\psi(\alpha/\gamma)$. This allows us to transform the study of the level set of fast Lyapunov exponents into the investigation of $E_{\psi}(\alpha)$.

We first give the necessary and sufficient conditions for $E_{\psi}(\alpha)$ to be non-empty.

\begin{lemma}\label{l.equivalent_increasing}
For any $0<\alpha < \infty$, $E_{\psi}(\alpha)$ is non-empty if and only if $\psi$ is equivalent to an increasing function.
\end{lemma}

\begin{proof}
Let $0<\alpha<\infty$ be fixed.
        For the ``only if" part, we assume that $E_{\psi}(\alpha)$ is non-empty. Take $x\in E_{\psi}(\alpha)$, and define
        \begin{equation*}
            \varphi(n)\coloneqq \frac{\sum^n_{k=1}\log a_k(x)}{\alpha},\quad \forall n \in\mathbb N.
        \end{equation*}
        Then $\varphi(n+1)\geq \varphi(n)$, and
        \begin{equation*}
                \lim_{n\to \infty}\frac{\varphi(n)}{\psi(n)}
                =\lim_{n\to \infty}\frac
                {\sum^n_{k=1}\log a_k(x)}{\alpha\psi(n)}
                =1.
        \end{equation*}
        This means that $\psi$ is equivalent to the increasing function $\varphi$.

For the ``if" part, suppose that $\psi$ is equivalent to an increasing function $\widehat{\varphi}$. Then
        \begin{equation*}
                \lim_{n\to \infty}\frac{\widehat{\varphi}(n)}{n}=
                \lim_{n\to \infty}\frac{\psi(n)}{n}\cdot
                \lim_{n\to \infty}\frac{\widehat{\varphi}(n)}{\psi(n)}=\infty.
        \end{equation*}
Let $s_1=t_1\coloneqq e^{1+\alpha\widehat{\varphi}(1)}$ and $s_n=t_n\coloneqq e^{1+\alpha(\widehat{\varphi}(n)-\widehat{\varphi}(n-1))}$ for all $n \geq 2$. Then $E(\{s_n\},\{t_n\})$ as defined in \eqref{e.def_E_s_n_t_n} is non-empty. Moreover, $E(\{s_n\},\{t_n\})$ is a subset of $E_\psi(\alpha)$. In fact, for any $x\in E(\{s_n\},\{t_n\})$, we deduce that
\[
e^{1+\alpha\widehat{\varphi}(1)} < a_1(x) \leq 2e^{1+\alpha\widehat{\varphi}(1)}
\]
and
\begin{equation*}\label{e.estimation_a_n}
                e^{1+\alpha(\widehat{\varphi}(n)-\widehat{\varphi}(n-1))}
                <{a_n}(x)
                \leq
                2e^{1+\alpha(\widehat{\varphi}(n)-\widehat{\varphi}(n-1))},\quad \forall n \geq 2.
        \end{equation*}
For any $n \in \mathbb N$, we have
\begin{equation*}
 \alpha \cdot\frac{\widehat{\varphi}(n)+n}
                {\widehat{\varphi}(n)}
               <
                \frac{\sum^n_{k=1}\log a_k(x)}{\widehat{\varphi}(n)}
                \leq
                \alpha \cdot\frac{\widehat{\varphi}(n)+(1+\log2)n}
                {\widehat{\varphi}(n)},
        \end{equation*}
which implies that
\[
\lim_{n\to \infty}  \frac{\sum^n_{k=1}\log a_k(x)}{\widehat{\varphi}(n)}=1.
\]
Thus,
\begin{equation*}
                \lim_{n\to \infty}\frac{\sum^n_{k=1}\log a_k(x)}{\psi(n)}
                =\lim_{n\to \infty}\frac{\sum^n_{k=1}\log a_k(x)}{\widehat{\varphi}(n)} \cdot \lim_{n\to \infty}\frac{\widehat{\varphi}(n)}{\psi(n)}
                =\alpha,
        \end{equation*}
that is, $x\in E_{\psi}(\alpha)$. Hence, $E(\{s_n\},\{t_n\}) \subset E_{\psi}(\alpha)$, and so $E_\psi(\alpha)$ is non-empty.
\end{proof}

In the following, we calculate the Hausdorff dimension of $E_{\psi}(\alpha)$ whenever it is non-empty.
By Lemma \ref{l.equivalent_increasing}, without loss of generality, we assume that $\psi:\mathbb{N}\to\mathbb{R}_{>0}$ is increasing. For any $0<\alpha<\infty$, we will show that
\begin{equation*}
\dim_{\rm H}E_{\psi}(\alpha)= \frac{1}{(\gamma-1)\beta_\psi+1},
 \end{equation*}
where
\[
  \beta_\psi \coloneqq \limsup_{n\to \infty }
    \frac{\psi(n+1)}{\psi(n)}.
\]
The proof is divided into two parts: the lower and upper bounds of $\dim_{\rm H}E_{\psi}(\alpha)$.

For the lower bound of  $\dim_{\rm H}E_{\psi}(\alpha)$, define two sequences $\{s_{n}\}_{n\geq 1}$ and $\{t_{n}\}_{n\geq 1}$ by putting $s_1=t_1\coloneqq e^{\alpha \psi(1)+1}$ and $s_n=t_n\coloneqq e^{\alpha(\psi(n)-\psi(n-1))+1}$ for all $n\geq 2$. Since $\psi(n)/n \to \infty$ as $n\to \infty$, we derive that $E(\{s_n\},\{t_n\})$ in \eqref{e.def_E_s_n_t_n} is a subset of $E_\psi(\alpha)$.
From Proposition \ref{p.Hausdorff_dimension_E_s_n_t_n}, we conclude that
\begin{alignat}{2}
 \dim_{\mathrm{H}}E_{\psi}(\alpha) \geq&\dim_{\mathrm{H}}\mathbb{E} (\{s_n\}\{t_n\})
        && \qquad \qquad \qquad\notag \\
        =&\liminf_{n\to \infty}\frac{\sum_{k=1}^n\log t_k}
        {\gamma \sum_{k=1}^{n+1}\log s_k-\log t_{n+1}} \notag\\
        =&\liminf_{n\to \infty}\frac{n+\alpha\psi(n)}
        {(\gamma-1)\alpha \psi(n+1)+\alpha\psi(n)+\gamma (n+1)-1} \notag \\
        =&\frac{1}{(\gamma-1)\beta_\psi+1}.\label{e.A_lower}
\end{alignat}

To establish the upper bound of $\dim_{\rm H}E_{\psi}(\alpha)$, we define, for any $x\in \Lambda$,
\begin{equation}\label{e.def_k}
        \mathrm{k} (x)\coloneqq\lim_{n\to \infty}
        \frac{\sum^n_{k=1}\log a_k(x)}{n}
\end{equation}
and
\begin{equation}\label{e.def_tau}
        \tau(x)\coloneqq \limsup_{n\to \infty}\frac
        {\sum^{n+1}_{k=1}\log a_k(x)}
        {\sum^n_{k=1}\log a_k(x)} =1+\limsup_{n\to \infty}
        \frac{\log a_{n+1}(x)}
        {\sum^n_{k=1}\log a_k(x)}.
\end{equation}
We claim that $E_{\psi}(\alpha)$ is a subset of the set of $x\in \Lambda$ for which $\tau(x)=\beta_\psi$ and $\mathrm{k} (x)=\infty$. To see this, for any $x\in E_{\psi}(\alpha)$, we have
\begin{equation*}
\lim_{n\to \infty}  \frac{\sum^n_{k=1}\log a_k(x)}{\psi(n)} =\alpha.
\end{equation*}
Since $\psi(n)/n \to \infty$ as $n\to \infty$, it follows from \eqref{e.def_k} that $\mathrm{k} (x)=\infty$. In view of \eqref{e.def_tau}, we derive that
\begin{equation*}
     \tau(x)\coloneqq  \limsup_{n\to \infty}\frac
        {\sum^{n+1}_{k=1}\log a_k(x)}
        {\sum^n_{k=1}\log a_k(x)}
        =\limsup_{n\to \infty}\frac{\psi(n+1)}{\psi(n)}
        =\beta_\psi.
\end{equation*}
Thus, the assertion holds.

For any $1\leq \beta \leq \infty$, let
 \begin{equation}\label{e.def_Gamma_infty_hat}
                \Gamma_{\infty}(\beta)\coloneqq
                \left\{x\in  \Lambda:\tau(x)\geq \beta,
                \mathrm{k} (x)=\infty\right\}.
        \end{equation}
Then
 \begin{equation}\label{RelEG}
 E_{\psi}(\alpha)\subset\Gamma_{\infty}(\beta).
\end{equation}
The following lemma gives an upper bound of the Hausdorff dimension of $\Gamma_{\infty}(\beta)$, which in turns provides the upper bound for $\dim_{\rm H}E_{\psi}(\alpha)$.

\begin{lemma}\label{l.Hausdorff_dimension_Gamma_infty} For any $1\leq \beta \leq \infty$, we have
        \begin{equation*}\label{e.upperboundGamma_infty}
\dim_{\mathrm{H}}{\Gamma}_{\infty}(\beta)
                \leq \frac{1}{(\gamma-1)\beta+1}.
        \end{equation*}
\end{lemma}

  \begin{proof}
  The proof proceeds by considering three cases: $\beta=1, 1<\beta <\infty$, and $\beta =\infty$.

\textbf{Case I: \bm{$\beta=1.$}} We observe that
\begin{equation*}
            {\Gamma}_{\infty}(1) = \Lambda_{\infty},
        \end{equation*}
        where $\Lambda_{\infty}$ is defined in \eqref{e.def_pi_infty}. It follows from Proposition \ref{l.Hausdorff_dimension_limsup_log_pi/n} that
        \begin{equation}\label{e.dim_Gamma<1/2}
                \dim_{\mathrm{H}}{\Gamma}_{\infty}(1)=\dim_{\mathrm{H}}\Lambda_{\infty}
                =\frac{1}{\gamma},
        \end{equation}
        which completes the proof for this case.

 \textbf{Case II: \bm{$1<\beta<\infty.$}} For any $0<\epsilon<\beta-1$, let $s\coloneqq \frac{1+2\epsilon}{(\gamma-1)(\beta-\epsilon)+1}$. Choose an integer $M\geq 2$ sufficiently large such that
        \begin{equation}\label{e.def_M}
                M^{\epsilon}> 2J_{\epsilon}C^s,
        \end{equation}
        where $J_{\epsilon}$ is defined in \eqref{e.def_J_epsilon}.
For any $n \in \mathbb N$, let
        \begin{equation}\label{e.def_B_n_ep_M}
                B_n(\epsilon,M)\coloneqq
                \Big\{
                x\in  \Lambda:
                a_{n+1}(x)>\big(
                        \prod_{k=1}^{n}
                        a_k(x)
                        \big)^{\beta-1-\epsilon}\ \text{and}\
                \prod_{k=1}^{n}
                a_k(x)\geq M^n
                \Big\}.
                \end{equation}
   Thus,
        \begin{equation}\label{e.cover_a_n}
                {\Gamma}_{\infty}(\beta)
                \subset \bigcap_{N=1}^{\infty}
                \bigcup_{n=N}^{\infty}
                B_n(\epsilon,M).
        \end{equation}
For any $n \in \mathbb N$, define
        \begin{equation}\label{e.def_D_n}
            \mathcal{D} _n(M)\coloneqq
        \left\{(i_1,\dots, i_{n})\in
        \mathbb{N}^n:
        i_1\cdots i_{n}\geq M^n\right\},
        \end{equation}
 and for any $(i_1,\dots, i_{n})
        \in \mathcal{D} _n(M)$, let
        \begin{equation}\label{e.def_J_n_another}
                J_n(i_1,\dots,i_{n}) \coloneqq
                \bigcup_{k>(i_1\cdots i_{n})
                ^{\beta-1-\epsilon}} I_{n+1}(i_1,\dots,i_{n},k).
        \end{equation}
Then $B_n(\epsilon,M)$ can be reformulated as
        \begin{equation}\label{e.a_n_J_n}
                B_n(\epsilon,M)=
                \bigcup_{(i_1,\dots, i_{n})
                \in \mathcal{D} _n(M)}
                J_n(i_1,\dots,i_{n}).
        \end{equation}
 Combining this with \eqref{e.cover_a_n}, for every $N\geq 1$, we see that,
        \begin{equation}\label{e.cover_by_J}
                {\Gamma}_{\infty}(\beta)\subset
                \bigcup^\infty_{n=N}\bigcup_{(i_0,\dots, i_{n-1})
                \in \mathcal{D}_n(M)}J_n(i_0,\dots,i_{n-1}).
        \end{equation}
For any $n \in \mathbb N$ and $(i_1,\dots, i_{n})\in \mathcal{D} _n(M)$, we estimate the diameter of $J_n(i_1,\dots,i_{n})$. Since
$$
\sum_{k>(i_1\cdots i_{n})
        ^{\beta-1-\epsilon}}\frac{1}{k^{\gamma}} \leq \frac{2^{\gamma-1}}
        {(\gamma-1)(i_1\cdots i_{n})^{(\beta-1-\epsilon)(\gamma-1)}},
$$
we deduce from Proposition \ref{p.cylinder} and \eqref{e.def_J_n_another} that
     \begin{alignat*}{2}
        |J_n(i_1,\dots,i_{n})|\leq\frac{C^{n+1}}
        {(i_1\cdots i_{n})^{\gamma}}
        \sum_{k>(i_1\cdots i_{n})
        ^{\beta-1-\epsilon}}\frac{1}{k^{\gamma}}
       \leq \frac{C^{n+1}2^{\gamma-1}}
        {(\gamma-1)(i_1\cdots i_{n})^{(\gamma-1)(\beta-\epsilon)+1}}.
        \end{alignat*}
Note that $i_1\cdots i_{n}\geq M^n$, we derive that
 \begin{equation*}\label{J^s}
|J_n(i_1,\dots,i_{n})|^s \leq \left(\frac{2^{\gamma-1}C}{\gamma-1}\right)^s \frac{C^{sn}}
        {(i_1\cdots i_{n})^{1+2\epsilon}} \leq \left(\frac{2^{\gamma-1}C}{\gamma-1}\right)^s \left(\frac{C^{s}}{M^\epsilon}\right)^n \frac{1}
        {(i_1\cdots i_{n})^{1+\epsilon}}.
 \end{equation*}
 Hence,
\begin{alignat*}{2}
 \sum_{(i_1,\dots, i_{n})\in \mathcal{D} _n(M)}|J_n(i_1,\dots,i_{n})|^s \leq &\left(\frac{2^{\gamma-1}C}{\gamma-1}\right)^s \left(\frac{C^{s}}{M^\epsilon}\right)^n \sum_{(i_1,\dots, i_{n})\in \mathcal{D} _n(M)} \frac{1}
        {(i_1\cdots i_{n})^{1+\epsilon}}\\
\leq & \left(\frac{2^{\gamma-1}C}{\gamma-1}\right)^s \left(\frac{J_\epsilon C^{s}}{M^\epsilon}\right)^n.
\end{alignat*}
From the definition of the $s$-dimensional Hausdorff measure, it follows that
 \begin{alignat*}{2}
 \mathcal{H}^{s}\left({\Gamma}_{\infty}(\beta)\right)
                \leq& \liminf_{N \to \infty}
                \sum_{n=N}^{\infty}
                \sum_{(i_1,\dots, i_{n})
                \in \mathcal{D} _n(M)}
                |J_n(i_1,\dots,i_{n})|^s \\
                \leq & \left(\frac{2^{\gamma-1}C}{\gamma-1}\right)^s \liminf_{N \to \infty}
                \sum_{n=N}^{\infty}\left(\frac{J_\epsilon C^{s}}{M^\epsilon}\right)^n\\
                 \leq & \left(\frac{2^{\gamma-1}C}{\gamma-1}\right)^s \liminf_{N \to \infty}
                \sum_{n=N}^{\infty} \frac{1}{2^n} =0. && \  \big(\text{by}\ \eqref{e.def_M}\big)
        \end{alignat*}
This implies that
 \begin{equation*}
                \dim_{\mathrm{H}}
              {\Gamma}_{\infty}(\beta)
                \leq s\coloneqq \frac{1+2\epsilon}{(\gamma-1)(\beta-\epsilon)+1}.
        \end{equation*}
 Letting $\epsilon\to 0^{+}$, we obtain $$\dim_{\mathrm{H}}{\Gamma}_{\infty}
        (\beta)\leq \frac{1}{(\gamma-1)\beta+1}.$$

\textbf{Case III: \bm{$\beta=\infty.$}} For any $K>1$, since
  \begin{equation*}
           {\Gamma}_{\infty}(\infty)\subset
           {\Gamma}_{\infty}(K),
        \end{equation*}
it follows from the result of Case II that
        \begin{equation*}
            \dim_{\mathrm{H}}{\Gamma}_{\infty}(\infty)
            \leq \dim_{\mathrm{H}}{\Gamma}_{\infty}(K)=\frac{1}{(\gamma-1)K+1}.
        \end{equation*}
Taking the limit as $K\to \infty$, we have $\dim_{\mathrm{H}}{\Gamma}_{\infty}(\infty)=0$.
\end{proof}

By Lemma \ref{l.Hausdorff_dimension_Gamma_infty}, we conclude that for any $0<\alpha<\infty$,
        \begin{equation*}\label{e.A_upper}
               \dim_{\mathrm{H}}E_{\psi}(\alpha)\leq \dim_{\mathrm{H}}{\Gamma}_{\infty}(\beta_\psi)=\frac{1}{(\gamma-1)\beta_\psi+1}.
        \end{equation*}
Combining this with \eqref{e.A_lower} and Proposition \ref{l.Transformation}, we complete the proof of Theorem \ref{t.Fast_Lyapunov} for the case $0<\alpha<\infty$.

\subsection{Case $\alpha =\infty$}
For any $x\in \Lambda$, we observe that
\[
\lim_{n\to \infty} \frac{\sum^n_{k=1}a_k(x)}{\psi(n)} =\infty \quad \Longleftrightarrow\quad \liminf_{n\to \infty} \frac{\sum^n_{k=1}a_k(x)}{\psi(n)} =\infty.
\]
This implies that $E_\psi(\infty)$ coincides with the set
\[
\left\{x\in \Lambda: \liminf_{n\to \infty} \frac{\sum^n_{k=1}a_k(x)}{\psi(n)} =\infty\right\}.
\]
The Hausdorff dimension of the latter will be given in Section \ref{Eunderline}. It follows from \eqref{Bpsiequal} that
\begin{equation*}
\dim_{\rm H}E_\psi(\infty)=\dim_{\rm H}\left\{x\in \Lambda: \liminf_{n\to \infty} \frac{\sum^n_{k=1}a_k(x)}{\psi(n)} =\infty\right\}=\frac{1}{(\gamma-1)B_\psi+1},
 \end{equation*}
where
\begin{equation*}
   B_\psi\coloneqq \limsup_{n\to \infty}\sqrt[n]{\psi(n)}.
\end{equation*}
We complete the proof of Theorem \ref{t.Fast_Lyapunov} for the case $\alpha =\infty$.

We end this section with some corollaries of Proposition \ref{p.Hausdorff_dimension_E_s_n_t_n}, Lemma \ref{l.Hausdorff_dimension_Gamma_infty} and Theorem \ref{t.Fast_Lyapunov}.

\begin{cor}\label{c.Hausdorff_dim_F}
        \begin{equation*}
                \dim_{\mathrm{H}}
                \left\{x\in \Lambda:\lim_{n \to \infty}
                a_n(x)=\infty\right\}
                =\frac{1}{\gamma}.
        \end{equation*}
\end{cor}

\begin{proof}
        Setting $\beta =1$ in Lemma \ref{l.Hausdorff_dimension_Gamma_infty}, we derive that
        \begin{equation*}
                \dim_{\mathrm{H}}\{x\in \Lambda:\lim_{n \to \infty}
                a_n(x)=\infty\}
                \leq \dim_{\mathrm{H}}\widehat{\Gamma}_{\infty}(1)=\frac{1}{\gamma}.
        \end{equation*}
Note that
        \begin{equation*}
         E(\{e^n\}, \{e^n\})
        \subset \{x\in \Lambda:\lim_{n \to \infty}
        a_n(x)=\infty\},
        \end{equation*}
 it follows from Corollary \ref{c.D} that
        \begin{equation*}
        \dim_{\mathrm{H}}\{x\in \Lambda:\lim_{n \to \infty}
        a_n(x)=\infty\}\geq
        \dim_{\mathrm{H}}
       E(\{e^n\},\{e^n\}) =\frac{1}{\gamma}.
        \end{equation*}
 \end{proof}

The following result gives a full description of the growth rate of the coding of the MR map.

\begin{cor}\label{c.growth_rate}
Let $\phi : \mathbb{N} \to \mathbb{R}_{>0}$ be a function such that $\phi(n) \to \infty$ as $n \to \infty$. Then
        \begin{equation}\label{e.growth_rate}
                \dim_{\mathrm{H}}
                \left\{x\in  \Lambda:\lim_{n\to \infty }
                \frac{\log a_n(x)}{\phi(n)}=1\right\}
                =\frac{1}{\gamma+(\gamma-1)\xi_\phi},
        \end{equation}
        where
        \begin{equation*}
                \xi_\phi\coloneqq \limsup_{n \to \infty}
                \frac{\phi(n+1)}{\phi(1)+\cdots+\phi(n)}.
        \end{equation*}
\end{cor}

\begin{proof}
The lower bound follows from Proposition \ref{p.Hausdorff_dimension_E_s_n_t_n} by letting $s_n=t_n \coloneqq 2e^{\phi(n)}$. For the upper bound, define $\widehat{\psi}(n)\coloneqq \sum_{k=1}^
        n \phi(k)$ for all $n \in \mathbb N$. Then $\widehat{\psi}$ is increasing and $\widehat{\psi}(n)/n\to \infty$ as $n \to \infty$. Moreover, for any $x\in \Lambda$, we have
        \begin{equation*}
                \lim_{n\to \infty}\frac{\log a_n(x)}{\phi(n)}=1
                 \quad \Longrightarrow \quad
                \lim_{n\to \infty}\frac{\sum^n_{k=1}\log a_k(x)}{\widehat{\psi}(n)}=1.
        \end{equation*}
        By Theorem \ref{t.Fast_Lyapunov}, we see that
        \begin{equation*}
                \dim_{\mathrm{H}}\left\{ x\in \Lambda:\lim_{n\to \infty}
                \frac{\log a_n(x)}{\phi(n)}=1 \right\}
                \leq \dim_{\mathrm{H}}E_{\widehat{\psi}}(1)
                =\frac{1}{(\gamma-1)\beta_{\widehat{\psi}}+1},
        \end{equation*}
   where
        \begin{equation*}
                \beta_{\widehat{\psi}}
                :=\limsup_{n\to\infty}
                \frac{\widehat{\psi}(n+1)}{\widehat{\psi}(n)}
                =1+\limsup_{n\to\infty}\frac{\phi(n+1)}{\phi(1)+\cdots +\phi(n)}
                =1+\xi_\phi.
        \end{equation*}
Hence,  \begin{equation*}
                \dim_{\mathrm{H}}
                \left\{x\in  \Lambda:\lim_{n\to \infty }
                \frac{\log a_n(x)}{\phi(n)}=1\right\}
                \leq \frac{1}{\gamma+(\gamma-1)\xi_\phi}.
        \end{equation*}
        We complete the proof.
\end{proof}

\section{Proof of Theorem \ref{t.Upper_and_lower_fast_Lyapunov_spectrum}}\label{s.proofoftheoremb}

In this section, we provide the proof of Theorem \ref{t.Upper_and_lower_fast_Lyapunov_spectrum}.
Assume that $\psi:\mathbb{N}\to\mathbb{R}_{>0}$ is a function such that $\psi(n)/n \to \infty$ as $n\to \infty$.

Since $J_\psi(0) \subset \overline{J}_\psi(0)$ and $J_\psi(0) \subset \underline{J}_\psi(0)$, it follows from Theorem \ref{t.Fast_Lyapunov} that $\overline{J}_\psi(0)$ and $\underline{J}_\psi(0)$ are of full Lebesgue measure. Hence, they have Hausdorff dimension one. This completes the proof of Theorem \ref{t.Upper_and_lower_fast_Lyapunov_spectrum} for the case $\alpha =0$.

For any $0< \alpha \leq \infty$, define
\begin{equation*}\label{e.def_E_psi_alpha}
        \overline{E}_{\psi}(\alpha)\coloneqq
        \left\{x\in \Lambda:\limsup_{n \to \infty}
        \frac{\sum_{k=1}^n \log a_{k}(x)}{\psi(n)}= \alpha \right\}
\end{equation*}
and
\[
  \underline{E}_{\psi}(\alpha)\coloneqq
        \left\{x\in \Lambda:\liminf_{n \to \infty}
        \frac{\sum_{k=1}^n \log a_{k}(x)}{\psi(n)}= \alpha \right\}.
\]
To prove Theorem \ref{t.Upper_and_lower_fast_Lyapunov_spectrum}, by Proposition \ref{l.Transformation}, it is sufficient to calculate the Hausdorff dimension of $\overline{E}_{\psi}(\alpha)$ and $\underline{E}_{\psi}(\alpha)$.

For any $b,c \in (1,\infty)$, define
\begin{equation*}
    \overline{D}(b,c)\coloneqq \left\{x\in \Lambda:\Pi_{n}(x)\geq b^{c^n}\
\mathrm{for\ infinitely\ many}\ n\in \mathbb N\right\},
\end{equation*}
 and
 \begin{equation*}
     \underline{D}(b,c)\coloneqq
     \left\{x\in \Lambda: \Pi_{n}(x)\geq b^{c^n}\
\mathrm{for\ sufficiently\ large}\ n\in \mathbb N\right\},
\end{equation*}
where $\Pi_n(x)$ is given by \eqref{Pin(x)}.

\begin{lemma}\label{l.Luczak}
Let $b,c \in (1,\infty)$ be fixed. For any $d\in (1,c)$, if $x\in \overline{D}(b,c)$, then
        \begin{equation}\label{e.Pi_n>_max_i.m.}
                \Pi_{n+1}(x)>
                \max \left\{
                (\Pi_n(x))^d,b^{d^{n+1}}\right\}
                \ \mathrm{for\ infinitely\ many}\ n\in \mathbb N.
        \end{equation}
\end{lemma}

\begin{proof}
Fix $b,c \in (1,\infty)$ and $d\in (1,c)$. For any $m\in \mathbb{N}$ and $x\in \overline{D}(a,c)$, since $d<c$, we can find an integer $k>m$ such that
        \begin{equation*}
                \Pi_{m}(x)<b^{c^kd^{m-k}}
                \quad \mathrm{and}\quad
                \Pi_k(x)\geq b^{c^k}.
        \end{equation*}
For any $n \in \mathbb N$, define $f(n)\coloneqq b^{c^kd^{n-k}}$. It follows that
        \begin{equation*}
                \Pi_{m}(x)<f(m)\quad \mathrm{and}\quad
                \Pi_{k}(x)\geq f(k).
        \end{equation*}
 Choose the largest $n$ such that $m\leq n<k$ and $\Pi_n(x)<f(n)$. Combining this with the observation that $f(n) > b^{d^{n}}$, we derive that
        \begin{equation*}
                \Pi_{n+1}(x)\geq f(n+1)=
                (f(n))^d>\max \left\{
                        (\Pi_n(x))^d,
                        b^{d^{n+1}}
                \right\}.
        \end{equation*}
        The proof is completed.
\end{proof}

The Hausdorff dimension of $\overline{D}(b,c)$ and $\underline{D}(b,c)$ is as follows.

\begin{lemma}\label{l.Hausdorff_dimension_pi_n_x>=a^c^n}
For any $b,c \in (1,\infty)$,
        \begin{equation*}
                \dim_{\mathrm{H}}
                \underline{D}(b,c)
                =\dim_{\mathrm{H}}
                \overline{D}(b,c)
                =\frac{1}{(\gamma-1)c+1}.
         \end{equation*}
\end{lemma}

\begin{proof}
Fix $b,c \in (1,\infty)$.
Since $\underline{D}(b,c) \subset \overline{D}(b,c)$, it suffices to calculate the lower bound of $\dim_{\mathrm{H}}\underline{D}(b,c)$ and the upper bound of $\dim_{\mathrm{H}}\overline{D}(b,c)$.

The lower bound of $\dim_{\mathrm{H}}\underline{D}(b,c)$ follows from Proposition \ref{p.Hausdorff_dimension_E_s_n_t_n} by setting $s_n=t_n \coloneqq 2b^{c^n}$. That is,
        \begin{equation*}
                \dim_{\mathrm{H}}
                \underline{D}(b,c)
                \geq \liminf _{n\to \infty}
                \frac{\sum_{k=1}^n\log (2b^{c^k})}
                {\gamma\sum_{k=1}^{n+1}\log (2b^{c^k})-\log (2b^{c^{n+1}})}
                =\frac{1}{(\gamma-1)c+1}.
        \end{equation*}

For the upper bound of $\dim_{\mathrm{H}}\overline{D}(b,c)$, we construct a cover of $\overline{D}(b,c)$ as follows. Fix $d\in (1,c)$ and $s\in (1/d,1)$. For any $x\in \overline{D}(b,c)$, by Lemma \ref{l.Luczak}, we deduce that
        \begin{equation*}
                \Pi_{n+1}(x)\geq
                \max \{
                        (\Pi_{n}(x))^d,
                        b^{d^{n+1}}
                \}
                >
                (\Pi_{n}(x))^{sd}b^{(1-s)d^{n+1}}
        \end{equation*}
for infinitely many $n\in \mathbb N$. This implies that
        \begin{equation}\label{limsupLambda}
                \overline{D}(b,c)\subset
                \bigcap_{N=1}^{\infty}
                \bigcup_{n=N}^{\infty}
               \left\{
                        x\in \Lambda:
                        a_{n+1}(x)>
                (\Pi_{n}(x))^{sd-1}b^{(1-s)d^{n+1}}
                \right\}.
        \end{equation}

For any $n \in \mathbb N$ and $(i_1,\cdots,i_{n}) \in \mathbb N^n$, let
 \begin{equation}\label{e.def_J_'}
                J_n(i_1,\dots,i_{n}) \coloneqq
                \bigcup_{j>
                (i_1\cdots i_{n})^{sd-1}
                b^{(1-s)d^{n+1}} }
                I_{n+1}(i,\dots,i_{n},j).
        \end{equation}
Then the limsup set on the right-hand side of \eqref{limsupLambda} can be written as
 \begin{equation*}
                \bigcap_{N=1}^{\infty}
                \bigcup_{n=N}^{\infty}
                \bigcup_{
                        (i_1,\dots, i_{n})
                        \in \mathbb{N}^n
                        }
                         J_n(i_1,\dots,i_{n}).
        \end{equation*}
 Hence, for every $N\geq 1$, we have
 \begin{equation}\label{e.cover_'}
                \overline{D}(b,c)\subset
                \bigcup_{n=N}^{\infty}
                \bigcup_{
                        (i_1,\cdots, i_{n})
                        \in \mathbb{N}^n
                        }
                        J_n(i_1,\dots,i_{n}).
        \end{equation}
which provides the desired covering.

For any $n \in \mathbb N$ and $(i_1,\cdots, i_{n})\in \mathbb{N}^n$, the diameter of $J_n(i_1,\dots,i_{n})$ is estimated as follows.
Note that
$$
\sum_{j>(i_1\cdots i_{n})^{sd-1}b^{(1-s)d^{n+1}}}\frac{1}{j^{\gamma}} \leq   \frac{2^{\gamma-1}}{\gamma-1} \cdot \frac{1}{(i_1\cdots i_{n})
                ^{(sd-1)(\gamma-1)}b^{(\gamma-1)(1-s)d^{n+1}}}.
$$
It follows from \eqref{e.def_J_'} and Proposition \ref{p.cylinder} that
\begin{alignat}{2}
              |J_n(i_1,\dots,i_{n})|=&\sum_{j>
                (i_1\cdots i_{n})^{sd-1}
                b^{(1-s)d^{n+1}}}
                |I_{n+1}(i_1,\dots,i_{n},j)|\notag \\
                \leq& \frac{C^{n+1}}{(i_1\cdots i_{n})^{\gamma}}\cdot \sum_{j>
                (i_1\cdots i_{n})^{sd-1}
               b^{(1-s)d^{n+1}}}
                \frac{1}{j^{\gamma}}\notag \\
                =&\frac{2^{\gamma-1}C^{n+1}}
                {(\gamma-1)(i_1\cdots i_{n})^{sd(\gamma-1)+1}
                b^{(\gamma-1)(1-s)d^{n+1}}}.\label{e.J_estimation}
        \end{alignat}
For any $0<\epsilon<1$, let $t\coloneqq \frac{1+\epsilon}{(\gamma-1)sd+1}$.
From \eqref{e.cover_'} and \eqref{e.J_estimation}, we conclude that
 \begin{alignat}{2}
\mathcal{H}^{t} \left(\overline{D}(b,c)\right)
                \leq& \liminf_{n\to\infty}
                \sum_{n=N}^{\infty}
                \sum_{(i_1,\dots, i_{n})\in
                \mathbb{N}^n}
               |J_n(i_1,\dots,i_{n})|^t \notag\\
                \leq& \liminf_{n\to\infty}
                \sum_{n=N}^{\infty}
                \sum_{(i_1,\cdots, i_{n})\in
                \mathbb{N}^n}
                \left(\frac{2^{\gamma-1}}{\gamma-1}\right)^t
                \frac{C^{(n+1)t}}
                {(i_1\cdots i_{n})^{((\gamma-1)sd+1)t}
                b^{(\gamma-1)(1-s)td^{n+1}}} \notag\\
                \leq&\liminf_{n\to\infty}
                \sum_{n=N}^{\infty}
                 \left(\frac{2^{\gamma-1}C}{\gamma-1}\right)^t
                \frac{(J_{\epsilon}C^{t})^n}{b^{(\gamma-1)(1-s)td^{n+1}}}=0.
                &&\notag
        \end{alignat}
 This implies that
$$\dim_{\mathrm{H}}\overline{D}(b,c)\leq t\coloneqq \frac{1+\epsilon}{(\gamma-1)sd+1}.$$
Letting $\epsilon \to 0^+$, $s\to 1^-$, and $d\to c^-$, we obtain
$$\dim_{\mathrm{H}}\overline{D}(b,c)\leq \frac{1}{(\gamma-1)c+1}.$$
\end{proof}

As a consequence of Proposition \ref{p.Hausdorff_dimension_E_s_n_t_n} and Lemma \ref{l.Hausdorff_dimension_pi_n_x>=a^c^n}, we have the {\L}uczak-type result for MR maps.

\begin{cor}
For any $b,c \in (1,\infty)$,
$$
\dim_{\rm H}\left\{x\in \Lambda:a_{n}(x)\geq b^{c^n}\
\mathrm{for\ infinitely\ many}\ n\in \mathbb N\right\}=\frac{1}{(\gamma-1)c+1}
$$
and
$$
\dim_{\rm H}\left\{x\in \Lambda:a_{n}(x)\geq b^{c^n}\
\mathrm{for\ sufficiently\ large}\ n\in \mathbb N\right\}=\frac{1}{(\gamma-1)c+1}.
$$
\end{cor}

We are ready to prove Theorem \ref{t.Upper_and_lower_fast_Lyapunov_spectrum} for the case  $0<\alpha \leq \infty$.

\subsection{Hausdorff dimension of $\overline{E}_{\psi}(\alpha)$}
For any $0<\alpha \leq \infty$, we will prove that
\begin{equation*}
\dim_{\rm H}\overline{E}_{\psi}(\alpha)=\frac{1}{(\gamma-1)b_\psi+1},
\end{equation*}
where
\begin{equation} \label{bproof}
   b_\psi\coloneqq \liminf_{n\to \infty}\sqrt[n]{\psi(n)}.
\end{equation}
The proof is divided into two parts: the upper and lower bounds of $\dim_{\rm H}\overline{E}_{\psi}(\alpha)$.

\subsubsection{Upper bound}
For any $x\in \overline{E}_{\psi}(\alpha)$, when $\alpha\in (0,\infty)$, we see that $\Pi_n(x)\geq e^{\frac{\alpha \psi(n)}{2}}$ for infinitely many $n\in \mathbb N$; when $\alpha=\infty$, we derive that $\Pi_n(x)\geq e^{\psi(n)}$ for infinitely many $n\in \mathbb N$. Hence,
 \begin{equation}\label{e.overline_E_cover}
        \overline{E}_{\psi}(\alpha)\subset
       \left\{x\in \Lambda:
        \Pi_{n}(x)\geq A^{\psi(n)}\
        \mathrm{for\ infinitely\ many}\ n\in \mathbb N\right\}
    \end{equation}
     for some $A>1$. This leads to the investigation of the Hausdorff dimension of the limsup set in \eqref{e.overline_E_cover}.

\begin{lemma}\label{l.Hausdorff_dimension_over_F_psi}
For any $A\in(1,\infty)$, define
     \begin{equation*}\label{e.def_over_F_psi}
        \overline{F}(\psi)\coloneqq
       \left\{
        x\in \Lambda:
        \Pi_{n}(x)\geq A^{\psi(n)}\
        \mathrm{for\ infinitely\ many}\ n\in \mathbb N
        \right\}.
     \end{equation*}
     Then
     \begin{equation*}
        \dim_{\mathrm{H}}\overline{F}(\psi)
        =\frac{1}{(\gamma-1)b_\psi+1}.
     \end{equation*}
\end{lemma}

\begin{proof}
The proof proceeds by considering the three cases: $b_\psi=1$, $1<b_\psi<\infty$, and $b_\psi=\infty$.

For the case $b_\psi=1$, on the one hand, since $\psi(n)/n\to \infty$ as $n\to \infty$, we get that $\overline{F}(\psi)$ is a subset of $\Lambda_{\infty}.$
   By Lemma \ref{l.Hausdorff_dimension_limsup_log_pi/n}, we see that
    \begin{equation*}
         \dim_{\mathrm{H}}\overline{F}(\psi)
         \leq  \dim_{\mathrm{H}} \Lambda_{\infty}
         =\frac{1}{\gamma}.
    \end{equation*}
On the other hand, for any $\epsilon>0$, by the definition of $b_\psi$ in \eqref{bproof}, we derive that $\psi(n)\leq (1+\epsilon)^n$
    for infinitely many $n\in \mathbb N$. This implies that
 \begin{equation*}
        \underline{D}(A,1+\epsilon)\subset \overline{F}(\psi).
    \end{equation*}
 It follows from Lemma \ref{l.Hausdorff_dimension_pi_n_x>=a^c^n}
    that $\dim_{\mathrm{H}}\overline{F}(\psi)\geq \frac{1}{(\gamma-1)(1+\epsilon)+1}$.
    Letting $\epsilon \to 0^{+}$, we obtain the desired lower bound.

 For the case $1<b_\psi<\infty$, let $0<\epsilon<b_\psi-1$. By the definition of $b_\psi$ in \eqref{bproof}, we see that $\psi(n)\leq (b_\psi+\epsilon)^{n}$ for infinitely many $n\in \mathbb N$ and that $\psi(n)\geq (b_\psi-\epsilon)^{n}$ for sufficiently large $n\in \mathbb N$. Hence,
\begin{equation*}
        \underline{D}(A,b_\psi+\epsilon)
        \subset
        \overline{F}(\psi)
        \subset
        \overline{D}(A,b_\psi-\epsilon).
    \end{equation*}
Applying Lemma \ref{l.Hausdorff_dimension_pi_n_x>=a^c^n}, we deduce that
    \begin{equation*}
        \frac{1}{(\gamma-1)(b_\psi+\epsilon)+1}\leq
        \dim_{\mathrm{H}}\overline{F}(\psi)
        \leq
        \frac{1}{(\gamma-1)(b_\psi-\epsilon)+1}.
    \end{equation*}
    Since $\epsilon$ is arbitrary, it follows that
    $$\dim_{\mathrm{H}}\overline{F}(\psi)=\frac{1}{(\gamma-1)b_\psi+1}.$$

  For the case $b_\psi=\infty$, let $K>1$. We have $\psi(n)>K^n$ for sufficiently large $n \in \mathbb N$,
    and so
    \begin{equation*}
        \overline{F}(\psi)\subset
        \overline{D}(A,K).
    \end{equation*}
 It follows from Lemma \ref{l.Hausdorff_dimension_pi_n_x>=a^c^n}
    that $$\dim_{\mathrm{H}}\overline{F}(\psi)\leq \frac{1}{K(\gamma-1)+1}.$$
   Taking the limit as $K\to \infty$, we obtain $\dim_{\mathrm{H}}\overline{F}(\psi)=0$.
    \end{proof}

   Combining \eqref{e.overline_E_cover} and Lemma \ref{l.Hausdorff_dimension_over_F_psi},
    we deduce that
    \begin{equation*}
   \dim_{\mathrm{H}}\overline{E}_{\psi}(\alpha)
         \leq \frac{1}{(\gamma-1)b_\psi+1},
    \end{equation*}
    and the estimation of the upper bound for $\dim_{\mathrm{H}}\overline{E}_{\psi}(\alpha)$ is completed.

\subsubsection{Lower bound}\label{constrction}
When $b_\psi=\infty$, it is clear that $\dim_{\mathrm{H}}\overline{E}_{\psi}(\alpha) \geq 0$ for all $0<\alpha \leq \infty$. When $1<b_\psi<\infty$, we will establish the lower bound of $\dim_{\rm H}\overline{E}_{\psi}(\alpha)$ in two cases: first for $0<\alpha<\infty$, then for $\alpha =\infty$.

In the following, we assume that $1<b_\psi<\infty$. When $0<\alpha<\infty$, we construct a subset $E(\{s_n\}, \{t_n\})$ of $\overline{E}_{\psi}(\alpha)$ and then use Proposition \ref{p.Hausdorff_dimension_E_s_n_t_n} to estimate the lower bound of $\dim_{\rm H}\overline{E}_{\psi}(\alpha)$. More precisely, for each $\epsilon>0$, we propose to construct an increasing function $g_{\psi}: \mathbb{N}\rightarrow \mathbb{R}_{>0}$ satisfying
    the following properties:
    \begin{enumerate}[(i).]
    \item \begin{equation*}\label{e.condition_i}
        \limsup_{n \to \infty}\frac{g_{\psi}(n)}{\psi (n)}=1;
    \end{equation*}
    \item \begin{equation*}\label{e.condition_ii}
        \lim_{n \to \infty}\frac{g_{\psi}(n)}{n}=\infty;
    \end{equation*}
    \item \begin{equation*}\label{e.condition_iii}
        \limsup_{n \to \infty}\frac{g_{\psi}(n+1)}{g_{\psi}(n)}
        \leq b_\psi+\epsilon.
    \end{equation*}
    \end{enumerate}
Let $s_1=t_1\coloneqq e^{\alpha g_{\psi}(1)}$, and for any $n \geq 2$, define $s_n=t_n \coloneqq e^{\alpha(g_{\psi}(n)-g_{\psi}(n-1))}$. We derive that $E(\{s_n\}, \{t_n\})$ is a subset of $\overline{E}_{\psi}(\alpha)$.
Note that
\[
\limsup_{n \to \infty} \frac{\log s_{n+1}}{\log s_1+\cdots +\log s_n} = \limsup_{n \to \infty}
         \frac{g_{\psi}(n+1)}{g_{\psi}(n)}-1 \leq b_\psi+\epsilon -1.
\]
It follows from Proposition \ref{p.Hausdorff_dimension_E_s_n_t_n} that
    \begin{alignat}{2}
        \dim_{\mathrm{H}}\overline{E}_{\psi}(\alpha)
         \geq&  \dim_{\mathrm{H}}E (\{s_n\},\{t_n\})\notag\\
         =& \liminf _{n\to \infty}
                \frac{\sum_{k=1}^n\log s_k}
                {\gamma\sum_{k=1}^{n+1}\log s_k-\log s_{n+1}}\notag\\
\geq& \frac{1}{(\gamma-1)(b_\psi+\epsilon)+1} . \label{lowalpha}
    \end{alignat}
Since $\epsilon>0$ is arbitrary, we obtain
    \begin{equation*}
        \dim_{\mathrm{H}}\overline{E}_{\psi}(\alpha)
         \geq  \frac{1}{(\gamma-1)b_\psi+1},
    \end{equation*}
which establishes the desired lower bound.

Next, we provide the precise construction of the function $g_\psi$. For for each $n,k \in \mathbb{N}$, let
\begin{equation}\label{e.def_c_n_k}
        c_{n,k}\coloneqq
        \begin{cases}
        \psi(k)(b_\psi+\epsilon)^{n-k},\quad &1\leq k\leq n; \\
        \psi(k), \quad &k\geq n+1,
        \end{cases}
    \end{equation}
and define
\begin{equation}\label{e.def_b_n}
        g_{\psi}(n)\coloneqq \inf_{k\geq 1}\{c_{n,k}\}=
        \inf\left\{\psi(1)(b_\psi+\epsilon)^{n-1},
        \dots,
        \psi(n-1)(b_\psi+\epsilon),
        \psi(n),
        \psi(n+1),
        \dots\right\}.
    \end{equation}
Since $\psi(n) \to \infty$ as $n\to \infty$, the infimum in \eqref{e.def_b_n} is achieved. The smallest index at which this infimum is achieved is denoted by $t_n$, namely
 \begin{equation}
        k_n\coloneqq
        \min \{k\geq 1: c_{n,k}=g_{\psi}(n)\}.
    \end{equation}
The following statements hold. For all $n,k \in \mathbb N$,
    \begin{enumerate}[(1).]
    \item $c_{n,k}\leq c_{n+1,k}\leq c_{n,k}(b_\psi+\epsilon)
    \ \mathrm{and}\ g_{\psi}(n)\leq g_{\psi}(n+1)\leq g_{\psi}(n)(b_\psi+\epsilon);$
    \item $k_{n+1}\geq k_n\ \mathrm{and}\ t_n\to \infty \ \mathrm{as}\  n \to \infty;$
    \item $g_{\psi}(n)\leq \psi(n)\  \mathrm{and}\ g_{\psi}(k_n)=\psi(k_n);$
    \item $g_{\psi}(n)/n\to \infty\ \mathrm{as}\ n \to \infty.$
    \end{enumerate}
Then these statements imply the properties of $g_{\psi}$. Precisely, items (2) and (3) imply Property (i); item (4) implies Property (ii); item (1) implies Property (iii).

In the following, we proceed with the proof of these statements, addressing each case individually.

For the item (1), given $n \in \mathbb N$, when $1\leq k\leq n$, we have
 \begin{equation*}
        \psi(k)(b_\psi+\epsilon)^{n-k}
        <\psi(k)(b_\psi+\epsilon)^{n+1-k}
       =\psi(k)(b_\psi+\epsilon)^{n-k} (b_\psi+\epsilon);
    \end{equation*}
   when $k\geq n+1$, we have
    \begin{equation*}
         \psi(k)
        = \psi(k)
         < \psi(k)(b_\psi+\epsilon).
    \end{equation*}
By the definition of $c_{n,k}$ in \eqref{e.def_c_n_k}, we derive that $c_{n,k}\leq c_{n+1,k}\leq c_{n,k}(b_\psi+\epsilon)$, and so their infimums satisfy the same relation:
    $g_{\psi}(n)\leq g_{\psi}(n+1)\leq g_{\psi}(n)(b_\psi+\epsilon)$.

 For the item (2), we first show the monotonicity of $\{k_n\}_{n\geq 1}$. Given $n\in \mathbb N$, when $k_n=1$, we have $k_{n+1}\geq 1= k_n$; when $2\leq k_n\leq n$, note that $i= k_n$ if and only if
    \begin{equation}\label{e.def_t_n}
        c_{n,i}<c_{n,k},\ \forall k<k_n
        \quad \mathrm{and} \quad
        c_{n, i}\leq c_{n,k},\ \forall k>k_n.
    \end{equation}
 For any $k<k_n$, we obtain
    \begin{equation*}
        \psi(k)(b_\psi+\epsilon)^{n-k}<\psi(k_n)(b_\psi+\epsilon)^{n-k_n},
    \end{equation*}
which implies that
    \begin{equation*}
        \psi(k)(b_\psi+\epsilon)^{n+1-k}
        <\psi(k_n)(b_\psi+\epsilon)^{n+1-k_n}.
    \end{equation*}
Hence, $k_{n+1}\geq k_n$; when $k_n\geq n+1$, for any $k<k_n$, we see that
    \begin{equation*}
        c_{n+1,k_n}=c_{n,k_n}< c_{n,k}\leq  c_{n+1,k}.
    \end{equation*}
It follows that $k_{n+1}\geq k_n$. Next, we use a proof by contradiction to show that $k_n\to \infty$ as $n \to \infty$. Suppose that $\lim_{n\to \infty}k_n\neq \infty$. Since $\{k_n\}_{n\geq 1}$ is increasing, there exists $N\in \mathbb N$ such that $k_n=N$ for sufficiently large $n \in \mathbb N$. Then for sufficiently large $n$ such that $n>N$, and
 \begin{equation*}
       \psi(n)=c_{n,n}\geq c_{n,k_n}=c_{n,N}
        =\psi(N)(b_\psi+\epsilon)^{n-N}.
    \end{equation*}
 It follows that
    \begin{equation*}\label{e.liminf_b}
    \frac{\log \psi(N)-N\log(b_\psi+\epsilon)}{n}
    +\log (b_\psi+\epsilon)\leq \frac{\log \psi(n)}{n}.
    \end{equation*}
Letting $n\to \infty$, we derive that
$$\liminf_{n\to \infty} \frac{\log \psi(n)}{n}\geq \log (b_\psi+\epsilon)>\log b_\psi,$$
which is a contradiction to the definition of $b_\psi$. Therefore, $k_n\to \infty \ \mathrm{as}\  n \to \infty$.

For the item (3), given $n\in \mathbb N$, by the definition of $g_\psi(n)$ in \eqref{e.def_b_n}, we see that $g_{\psi}(n)
    \leq c_{n,n}=\psi(n)$. Next, we show that $g_{\psi}(k_n)=\psi(k_n)$.
If $k_n<n$, then for all $1\leq k< k_n$,
    \begin{equation*}
        c_{t_n,k}
        =\psi(k)(b_\psi+\epsilon)^{n-k}(b+\epsilon)^{k_n-n}
        =c_{n,k}(b_\psi+\epsilon)^{k_n-n}
        >c_{n,k_n}(b_\psi+\epsilon)^{k_n-n}
        =c_{k_n,k_n};
    \end{equation*}
for all $k_n<k\leq n$,
    \begin{equation*}
        c_{k_n,k}= c_{n,k}(b_\psi+\epsilon)^{k-n}
        \geq c_{n,k_n}(b_\psi+\epsilon)^{k-n}
        = c_{k_n,k_n}(b_\psi+\epsilon)^{k-k_n}
        >c_{k_n,k_n};
    \end{equation*}
 and for all $k>n$,
    \begin{equation*}
        c_{k_n,k}= c_{n,k}
        \geq c_{n,k_n}
        = c_{k_n,k_n} (b_\psi+\epsilon)^{n-k_n}
        >c_{k_n,k_n}.
    \end{equation*}
  It follows from \eqref{e.def_t_n} that $k_{k_n}=k_n$ and $g_{\psi}(k_n)=\psi(k_n)$. If $k_n=n$, then
    \begin{equation*}
       g_{\psi}(k_n)=g_{\psi}(n)=c_{n,k_n}=c_{n,n}=\psi(n)=\psi(k_n).
    \end{equation*}
If $k_n>n$, then
    for all $1\leq k\leq n$,
    \begin{equation*}
        c_{k_n,k}= \psi(k)(b_\psi+\epsilon)^{k_n-k}
        \geq \psi(k)(b_\psi+\epsilon)^{n-k}
        =c_{n,k}
        >c_{n,k_n}
        =c_{k_n,k_n};
    \end{equation*}
    for all $n<k\leq k_n$,
    \begin{equation*}
        c_{k_n,k}= \psi(k)(b_\psi+\epsilon)^{k_n-k}
        \geq \psi(k)
        =c_{n,k}
        >c_{n,k_n}
        =c_{k_n,k_n};
    \end{equation*}
    and for all $k>k_n$,
    \begin{equation*}
        c_{t_n,k}=\psi(k)
        =c_{n,k}
        \geq c_{n,t_n}
        =c_{t_n,t_n}.
    \end{equation*}
    Thus $g_{\psi}(k_n)=\psi(k_n)$.

 For the item (4), since \begin{equation*}
        \lim_{n\to \infty}\frac{\psi(n)}{n}=\lim_{n\to \infty}\frac{(b_\psi+\epsilon)^n}{n}=\infty,
    \end{equation*}
    we have
    \begin{equation}\label{e.c_n_to_infty}
        \lim_{n\to \infty}\min_{1\leq k\leq n}\left\{
         \frac{\psi(k)(b_\psi+\epsilon)^{n-k}}{n}\right\}=\infty.
    \end{equation}
  Denote $c_n \coloneqq \min_{1\leq k\leq n}\{\frac{\psi(k)(b_\psi+\epsilon)^{n-k}}{n}\}$. It follows that
    \begin{equation*}
        \frac{g_{\psi}(n)}{n}=\inf\left\{
        c_n,
        \frac{\psi(n)}{n},
        \frac{\psi(n+1)}{n},
        \cdots\right\}
        \geq
        \inf\left\{
        c_n,
        \frac{\psi(n)}{n},
        \frac{\psi(n+1)}{n+1},
        \cdots\right\}.
    \end{equation*}
Combining this with \eqref{e.c_n_to_infty} and the hypothesis that $\psi(n)/\to \infty$ as $n \to \infty$, we derive that $g_{\psi}(n)/n\to \infty$ as $n \to \infty$.

At last, we deal with the lower bound of $\dim_{\rm H}\overline{E}_{\psi}(\alpha)$ for the case $\alpha =\infty$. For any $n \in \mathbb N$, define $\psi^*(n)\coloneqq n\psi(n)$. Then
\begin{equation*}
   b_{\psi^*}\coloneqq \liminf_{n\to \infty}\sqrt[n]{\psi^*(n)} = \liminf_{n\to \infty}\sqrt[n]{n\psi(n)}= \liminf_{n\to \infty}\sqrt[n]{\psi(n)} =b_\psi.
\end{equation*}
Note that $\overline{E}_{\psi^*}(1) \subset \overline{E}_{\psi}(\infty)$, it follows from \eqref{lowalpha} that
\[
\dim_{\rm H} \overline{E}_{\psi}(\infty) \geq \dim_{\rm H}\overline{E}_{\psi^*}(1) \geq \frac{1}{(\gamma-1)b_{\psi^*}+1} = \frac{1}{(\gamma-1)b_{\psi}+1}.
\]

\subsection{Hausdorff dimension of $\underline{E}_{\psi}(\alpha)$}\label{Eunderline}

For any $0<\alpha \leq \infty$, we will prove that
\begin{equation}\label{Bpsiequal}
\dim_{\rm H}\underline{E}_{\psi}(\alpha)=\frac{1}{(\gamma-1)B_\psi+1},
\end{equation}
where
\begin{equation*} \label{bproof}
   B_\psi\coloneqq \limsup_{n\to \infty}\sqrt[n]{\psi(n)}.
\end{equation*}

\subsubsection{Upper bound}
For any $x\in \underline{E}_{\psi}(\alpha)$, when $\alpha\in (0,\infty)$, we see that $\Pi_n(x)\geq e^{\frac{\alpha \psi(n)}{2}}$ for sufficiently large $n \in \mathbb N$; when $\alpha=\infty$, we see that $\Pi_n(x)\geq e^{\psi(n)}$ for sufficiently large $n \in \mathbb N$. Hence,
\begin{equation}\label{e.underline_E_cover}
\underline{E}_{\psi}(\alpha)\subset\left\{x\in \Lambda:\Pi_{n}(x)\geq A^{\psi(n)}\ \text{for sufficiently large $n \in \mathbb N$} \right\}
\end{equation}
for some $A>1$. This leads to the study of the Hausdorff dimension of the liminf set in \eqref{e.underline_E_cover}.

\begin{lemma}\label{l.Hausdorff_dimension_under_F_psi}
     For any $A\in(1,\infty)$, define
     \begin{equation*}\label{e.def_under_F_psi}
        \underline{F}(\psi)\coloneqq
        \left\{
        x\in \Lambda:
        \Pi_{n}(x)\geq A^{\psi(n)}
       \ \text{for sufficiently large $n \in \mathbb N$}
        \right\}.
     \end{equation*}
     Then
     \begin{equation*}
        \dim_{\mathrm{H}}\underline{F}(\psi)
        =\frac{1}{(\gamma-1)B_\psi+1}.
     \end{equation*}
\end{lemma}

\begin{proof}
The proof is very similar to that of Lemma \ref{l.Hausdorff_dimension_over_F_psi}.

  For the case $B_\psi=1$, since $\psi(n)/n\to \infty$ as $n\to \infty$,
    we get that $\underline{F}(\psi)$ is a subset of $\Lambda_{\infty}.$
    By Lemma \ref{l.Hausdorff_dimension_limsup_log_pi/n},
    \begin{equation*}
         \dim_{\mathrm{H}}\underline{F}(\psi)
         \leq  \dim_{\mathrm{H}} \Pi_{\infty}
         =\frac{1}{\gamma}.
    \end{equation*}
  For any $\epsilon>0$, by definition of $B_\psi$, we obtain $\psi(n)\leq (1+\epsilon)^n$ for sufficiently large $n \in \mathbb N$. Hence,
 \begin{equation*}
        \underline{D}(A,1+\epsilon)\subset \underline{F}(\psi).
    \end{equation*}
 It follows from Lemma \ref{l.Hausdorff_dimension_pi_n_x>=a^c^n} that $$\dim_{\mathrm{H}}\underline{F}(\psi)\geq \frac{1}{(\gamma-1)(1+\epsilon)+1}.$$ Letting $\epsilon \to 0^{+}$, we get the desired lower bound.

 For the case $1<B_\psi<\infty$, let $0<\epsilon<B_\psi-1$. By definition of $B_\psi$, we derive that $\psi(n)\leq (B_\psi+\epsilon)^{n}$ for sufficiently large $n \in \mathbb N$ and that $\psi(n)\geq (B_\psi-\epsilon)^{n}$ for infinitely many $n \in \mathbb N$. Thus,
 \begin{equation*}
        \underline{D}(A,B_\psi+\epsilon)
        \subset
        \underline{F}(\psi)
        \subset
        \overline{D}(A,B_\psi-\epsilon).
    \end{equation*}
Applying Lemma \ref{l.Hausdorff_dimension_pi_n_x>=a^c^n}, we have
    \begin{equation*}
        \frac{1}{(\gamma-1)(B_\psi+\epsilon)+1}\leq
        \dim_{\mathrm{H}}\underline{F}(\psi)
        \leq
        \frac{1}{(\gamma-1)(B_\psi-\epsilon)+1}.
    \end{equation*}
    Since $\epsilon>0$ is arbitrary, we obtain
    $$\dim_{\mathrm{H}}\underline{F}(\psi)=\frac{1}{(\gamma-1)B_\psi+1}.$$

For the case $B_\psi=\infty$, given $K>1$, we have $\psi(n)>K^n$ for infinitely many $n \in \mathbb N$, and so
  \begin{equation*}
        \underline{F}(\psi)\subset
        \overline{D}(A,K).
    \end{equation*}
We deduce from Lemma \ref{l.Hausdorff_dimension_pi_n_x>=a^c^n} that $$\dim_{\mathrm{H}}\underline{F}(\psi)\leq \frac{1}{(\gamma-1)K+1}.$$
Letting $C\to \infty$, we obtain $\dim_{\mathrm{H}}\underline{F}(\psi)=0$.
\end{proof}

 Combining \eqref{e.underline_E_cover} and Lemma \ref{l.Hausdorff_dimension_under_F_psi}, we conclude that
    \begin{equation*}
         \dim_{\mathrm{H}}\underline{E}_{\psi}(\alpha)
         \leq \frac{1}{(\gamma-1)B_\psi+1} .
    \end{equation*}

\subsubsection{Lower bound}\label{constrction'}
When $0<\alpha <\infty$, for any $\epsilon>0$, Fang, Shang and Wu in \cite{FSW} constructed a sequence $\{d_n\}_{n\geq 1}$ by $d_1\coloneqq T_1$ and $d_n\coloneqq T_n/T_{n-1}$ for $n \geq 2$, where $T_n\coloneqq \sup_{k\geq 1}\{d_{n,k}\}$ and $d_{n,k}$ is defined as
\begin{equation*}
        d_{n,k}\coloneqq
        \begin{cases}
        e^{\alpha\psi(k)},\quad &1\leq k\leq n; \\
        e^{\alpha\psi(k)(B_\psi+\epsilon)^{n-k}}, \quad &k\geq n+1.
        \end{cases}
    \end{equation*}
 They proved that $\{d_n\}_{n\geq 1}$ satisfies
    \begin{enumerate}[(i)]
    \item \begin{equation*}\label{e.condition_i'}
        \liminf_{n \to \infty}\frac{\sum^n_{k=1}\log c_k}{\psi (n)}=\alpha;
    \end{equation*}
    \item \begin{equation*}\label{e.condition_ii'}
        \lim_{n \to \infty}\frac{\sum^n_{k=1}\log c_k}{n}=\infty;
    \end{equation*}
    \item \begin{equation*}\label{e.condition_iii'}
        \limsup_{n \to \infty}\frac{\log c_{n+1}}{\sum^n_{k=1}\log c_k}
        \leq B_\psi+\epsilon-1.
    \end{equation*}
    \end{enumerate}
See \cite[p.\,828--829]{FSW} for a detailed proof. For any $n\in \mathbb N$, let $s_n=t_n\coloneqq 2c_n$. Then $E(\{s_n\},\{t_n\})$ is a subset of $\underline{E}_\psi(\alpha)$. By Proposition \ref{p.Hausdorff_dimension_E_s_n_t_n}, we derive that
\begin{align*}
        \dim_{\mathrm{H}}\underline{E}_{\psi}(\alpha)
         \geq&  \dim_{\mathrm{H}}E(\{s_n\},\{t_n\})\\
         =& \liminf _{n\to \infty}
                \frac{\sum_{k=1}^n\log c_k}
                {\gamma\sum_{k=1}^{n+1}\log c_k-\log c_{n+1}}\\
         =& \left(\gamma+(\gamma-1)\cdot\limsup_{n \to \infty}
         \frac{\log c_{n+1}}{\sum^n_{k=1}\log c_k}\right)^{-1}\\
         \geq& \frac{1}{(\gamma-1)(B_\psi+\epsilon)+1}.
    \end{align*}
   Since $\epsilon>0$ is arbitrary, we have
    \begin{equation*}
        \dim_{\mathrm{H}}\underline{E}_{\psi}(\alpha)
         \geq  \frac{1}{(\gamma-1)B_\psi+1}.
    \end{equation*}

When $\alpha =\infty$, for any $n \in \mathbb N$, define $\psi^*(n)\coloneqq n\psi(n)$. Then
\begin{equation*}
   B_{\psi^*}\coloneqq \limsup_{n\to \infty}\sqrt[n]{\psi^*(n)} =  \limsup_{n\to \infty}\sqrt[n]{\psi(n)} =B_\psi.
\end{equation*}
Since $\underline{E}_{\psi^*}(1) \subset \underline{E}_{\psi}(\infty)$, we deduce from \eqref{lowalpha} that
\[
\dim_{\rm H} \underline{E}_{\psi}(\infty) \geq \dim_{\rm H}\underline{E}_{\psi^*}(1) \geq \frac{1}{(\gamma-1)B_{\psi^*}+1} = \frac{1}{(\gamma-1)B_{\psi}+1}.
\]

\section{Appendix}\label{s.appendix}
This appendix presents a detailed proof for Theorem \ref{thm-1}. Throughout the proof, we construct a non-decreasing function $g_{\psi}:\mathbb{N}\to\mathbb{R}_{>0}$ that exhibits properties stronger than those outlined in Properties (i), (ii), and (iii) in Subsection \ref{constrction}. Although $g_{\psi}$ is not used in the proofs of our main theorems, we believe the construction of $g_{\psi}$ is of independent interest. This has prompted us to provide a thorough proof.

\begin{theorem}\label{thm-1}
Suppose that a function $\psi : \mathbb{N} \to \mathbb{R}_{>0}$ satisfies the following two hypotheses
\begin{equation*}\label{equ-1}
(A). \lim_{n\to \infty} \frac{\psi(n)}{n} = \infty, \quad (B). \liminf_{n\to \infty} \sqrt[n]{\psi(n)} = b_{\psi} \in [1, \infty],
\end{equation*}
then there exists a non-decreasing function $g_\psi : \mathbb{N} \to \mathbb{R}_{>0}$ satisfying the following four properties:
\begin{equation*}\label{equ-2}
(a). g_\psi(n) \leq \psi(n);~~(b). \limsup_{n\to \infty} \frac{g_\psi(n)}{\psi(n)} = 1;~~ (c). \lim_{n\to \infty} \frac{g_\psi(n)}{n} = \infty;~~(d). \lim_{n\to \infty} \frac{g_\psi(n+1)}{g_\psi(n)} = b_{\psi}.
\end{equation*}
\end{theorem}
Depending on the different choices of $b$ of Hypothesis $(B)$ in \eqref{equ-1}, we will split the proof of Theorem \ref{thm-1} into the following three propositions, namely Propositions \ref{thm-2}, \ref{thm-3} and \ref{thm-4}. Each of these propositions corresponds to the cases $b_{\psi}=\infty$, $1<b_{\psi}<\infty$ and $b_{\psi}=1$ respectively.
\begin{prop}[$b_{\psi}=\infty$]\label{thm-2}
Suppose that a function $\psi: \mathbb{N} \to \mathbb{R}_{>0}$ satisfies $\lim\limits_{n\to\infty}\psi(n)^{1/n}=\infty$, then there is a non-decreasing function $g_\psi: \mathbb{N} \to \mathbb{R}_{>0}$, satisfying
\begin{equation*}\label{equ-3}
(a). g_\psi(n) \leq \psi(n); \quad (b). \limsup_{n\to \infty}\frac{g_\psi(n)}{\psi(n)}=1; \quad (c). \lim_{n\to \infty} \frac{g_\psi(n+1)}{g_\psi(n)}=\infty.
\end{equation*}
\end{prop}
\begin{prop}[$1<b_{\psi}<\infty$]\label{thm-3}
Suppose that a function $\psi: \mathbb{N} \to \mathbb{R}_{>0}$ satisfies $\liminf\limits_{n\to\infty}\psi(n)^{1/n}=b_{\psi}\in (1,\infty)$, then there is a non-decreasing function $g_\psi: \mathbb{N} \to \mathbb{R}_{>0}$ satisfying
\begin{equation*}\label{equ-4}
(a). g_\psi(n) \leq \psi(n); \quad (b). \limsup_{n\to \infty} \frac{g_\psi(n)}{\psi(n)} = 1; \quad (c). \lim_{n\to \infty} \frac{g_\psi(n+1)}{g_\psi(n)} = b_{\psi}.
\end{equation*}
\end{prop}
\begin{prop}[$b_{\psi}=1$]\label{thm-4}
Suppose  that a function $\psi : \mathbb{N} \to \mathbb{R}_{>0}$ satisfies the following hypothesis
\begin{equation*}\label{equ-5}
(A). \lim_{n\to \infty} \frac{\psi(n)}{n} = \infty; \quad (B). \liminf_{n\to \infty} {\psi(n)}^{1/n} = 1,
\end{equation*}
then there is a non-decreasing function $g_\psi : \mathbb{N} \to \mathbb{R}_{>0}$ satisfying
\begin{equation*}\label{equ-6}
(a). g_\psi(n) \leq \psi(n); \quad (b). \limsup_{n\to \infty}\frac{g_\psi(n)}{\psi(n)} = 1; \quad (c). \lim_{n\to \infty} \frac{g_\psi(n)}{n} = \infty; \quad (d)\lim_{n\to \infty} \frac{g_\psi(n+1)}{g_\psi(n)} = 1.
\end{equation*}
\end{prop}

When $b_{\psi}=\infty$ or $1<b_{\psi}<\infty$,  Hypothesis $(A)$ in Theorem \ref{thm-1} is deduced from Hypothesis $(B)$; While Property $(c)$ is deduced from Property $(d)$. Hence, these two items are omitted in the statements of Propositions \ref{thm-2} and \ref{thm-3}.
		
\subsection{Key Lemmas}\label{subs.lemcor}
In this subsection, we will prove the following two Lemmas \ref{cor-1} and \ref{cor-2}, which play an important role in proving Propositions \ref{thm-3} and \ref{thm-4}.

\begin{lemma}[Key Lemma 1]\label{cor-1}
Suppose that a function $\psi:\mathbb{N}^{*}\to\mathbb{R}_{+}$ satisfies the assumptions in Proposition \ref{thm-3}, then for every $\epsilon \in (0, b_{\psi}-1)$ and $N\in\mathbb{N}$, there is an integer $n^*>N$, such that
\begin{equation}\label{equ-8}
\psi(n)>(b_{\psi}-\epsilon)^{n-n^*} \psi(n^{*}), \forall n>n^* \quad and \quad \psi(n) > (b_{\psi}+\epsilon)^{n-n^{*}} \psi(n^*), \forall n <n^*.
\end{equation}
\end{lemma}
and
\begin{lemma}[Key Lemma 2]\label{cor-2}
Suppose that a function $\psi:\mathbb{N}^{*}\to\mathbb{R}_{+}$ satisfies the assumptions in Proposition \ref{thm-4}, then for every $\epsilon>0$ and $N\in\mathbb{N}$, there is an integer $n^*>N$, such that
\begin{equation}\label{equ-9}
n^* \psi(n)>n \psi(n^*), \forall n>n^* \quad and \quad \psi(n) > (1+\epsilon)^{n-n^*} \psi(n^*), \forall n<n^*.
\end{equation}
\end{lemma}
To proceed with the proof of Lemmas \ref{cor-1} and \ref{cor-2}, let us introduce some notions. Given a sequence of positive real numbers $\{a(n)\}_{n\in\mathbb{N}}$, denote an index $\hat{n}\in \mathbb{N}$ as an \emph{L-index} of $\{a(n)\}_{n\in\mathbb{N}}$, if $a(n) > a(\hat{n}),\forall n>\hat{n}$. Suppose $\{a(n)\}_{n\in\mathbb{N}}$ admits infinitely many L-indices, then one can order \textbf{all} L-indices into an increasing sequence of natural numbers $0=n_0<n_1<n_2<\cdots<n_k<\cdots$, and every $n\in (n_{j},n_{j+1})$ is not an L-index. Such a sequence is unique and is denoted as the \emph{maximal infinite L-sequence} of $\{a(n)\}_{n\in\mathbb{N}}$.

Analogously, denote an index $\bar{n}\in \mathbb{N}$ as a \emph{P-index} of $\{a(n)\}_{n\in\mathbb{N}}$, if $a(n) > a(\bar{n}), \forall n<\bar{n}$. Suppose $\{a(n)\}_{n\in\mathbb{N}}$ admits infinitely many P-indices, one can well define the \emph{maximal infinite P-sequence} of $\{a(n)\}_{n\in\mathbb{N}}$ accordingly.

\medskip
	
With the convention of maximal L-sequence, we have
\begin{lemma}\label{lem-2}
Given a sequence $\{a(n)\}_{n\in\mathbb{N}}$, suppose $0=n_0 < n_1 < n_2 < \cdots$ is the infinite maximal L-sequence of $\{a(n)\}_{n\in\mathbb{N}}$, then we have the following two properties:
\begin{itemize}
\item[(i).] the subsequence $\{a(n_k)\}_{k=1}^\infty$ over the maximal infinite L-sequence is strictly ascending;
\item[(ii).] for every $k\in\mathbb{N}$ and every $m\in (n_{k-1}, n_k)$ (if any), one has $a(m) \geq a(n_k)$.
\end{itemize}
\end{lemma}
\begin{proof}[Proof of Lemma \ref{lem-2}]
First, Property $(i)$ follows directly from the definition of the maximal infinite L-sequence.

Next, we proceed the proof of Property $(ii)$. Without loss of generality, fix an integer $k>0$ with $n_{k-1}<n_{k}-1$. Otherwise, if no such integer $k$ exists, there is nothing to prove. We will show that Property $(ii)$ holds for all $m\in (n_{k-1}, n_k)$. To this end, we will use induction on $m$ (via a revise order).

First, let us show Property $(ii)$ holds when $m=n_{k}-1$. We proceed the proof by contradiction. Suppose on the contrary, we have
\begin{equation}\label{equ:n_k-1}
a(n_{k}-1)<a(n_{k}).
\end{equation}
On the other hand, due to the fact that $n_{k}$ is an L-index, then for every $n\geq n_{k}$, we have
\begin{equation}\label{equ:n_k}
a(n)>a(n_{k}).
\end{equation}
Combining \eqref{equ:n_k-1} and \eqref{equ:n_k} together, it follows that
\begin{equation}\label{equ:n_k-12}
a(n_{k}-1)<a(n),~\mbox{for every}~~n>n_{k}-1.
\end{equation}
In other words, \eqref{equ:n_k-12} means that $n_{k}-1$ is also an L-index, which is obviously a contradiction to the fact that $\{n_{k}\}$ is the maximal L-sequence. Therefore we have $a(n_{k}-1)\geq a(n_{k})$.

Next, suppose $m,m-1\in (n_{k-1},n_{k})$, and we have already show Property $(ii)$ holds for $m$, then we have
\begin{equation}\label{equ:m-1}
a(n)\geq a(n_{k}),~~\mbox{for every}~~n\in[m,n_{k}).
\end{equation}
Meanwhile, suppose on the contrary that Property $(ii)$ fails for $m-1$, then
\begin{equation}\label{equ:m}
a(m-1)<a(n_{k}),
\end{equation}
Combining \eqref{equ:m-1} and \eqref{equ:m}, if follows that
\begin{equation}\label{equ:m-12}
a(m-1)<a(n_{k})\leq a(n),~~\mbox{for every}~~n\in[m,n_{k}).
\end{equation}
Together with \eqref{equ:n_k}, it further implies that
\begin{equation}
a(m-1)<a(n),~~\mbox{for every}~~n>m-1,
\end{equation}
which means that $m-1$ is also an L-index. Obviously, it is a contradiction to the fact that $\{n_{k}\}$ is the maximal infinite L-sequence. Therefore we complete the proof of Lemma \ref{lem-2}.
\end{proof}

\begin{lemma}\label{lem-7.5}
The following two properties are true.
\begin{itemize}
\item[(i)] A sequence $\{a(n)\}_{n\in\mathbb{N}}$ of real numbers satisfies $\lim\limits_{n\to\infty}a(n)=\infty$ admits a maximal infinite L-sequence;
\item[(ii)]A sequence $\{c(n)\}_{n\in\mathbb{N}}$ of positive real numbers satisfies $\liminf\limits_{n\to \infty} c(n) = 0$ admits a maximal infinite P-sequence.
\end{itemize}
\end{lemma}
\begin{proof}[Proof of Lemma \ref{lem-7.5}]
Let us first proceed the proof of Property (i). For each $n\in\mathbb{N}$, denote by $b(n):= \inf\limits_{k>n} a(k)$. Clearly, $b(n)<\infty$ and is increasing for all $n\in \mathbb{N}$. Moreover, the hypothesis that $\lim\limits_{n\to\infty}a(n)=\infty$ implies that $\limsup\limits_{n\to\infty}b(n)=\infty$. Hence, there are infinitely many indices $n_1<n_2<\cdots<n_k<\cdots$ such that $b(n_k)>b(n_k - 1)$. Together with the definition of $\{b(n)\}_{n\in\mathbb{N}}$,  it indicates that $a_{n_{k}}=\min_{n\geq n_k}\{a(n)\}$, which yields that $a(n)>a(n_k)$ for all $n> n_k$. Hence the sequence $\{a(n)\}_{n\in\mathbb{N}}$ has infinitely many L-indices. Therefore, $\{a(n)\}_{n\in\mathbb{N}}$ admits a maximal infinite L-sequence, as we want.
			
Analogously, for Property (ii), denote by $d(m):=\min\limits_{1\leq n \leq m} c(n)$. Then $\{d(m)\}_{m\in\mathbb{N}}$ is a positive decreasing sequence and  $\lim\limits_{m\to\infty}d(m)=0$. Hence there are infinitely many indices $m_1<m_2<\cdots<m_k<\cdots$, such that $d(m_k) < d(m_{k}-1)$. In other words, $c(m_k) < c(n)$ for every $n<m_k$. Consequently, the sequence $\{c(n)\}_{n\in\mathbb{N}}$ admits a maximal infinite P-sequence, as we want.
\end{proof}

Based on Lemma \ref{lem-2}, we will prove the following lemma, which is a cornerstone for verifying Lemma \ref{cor-1} and \ref{cor-2}.
		
\begin{lemma}\label{lem-3}
Suppose that $\{a(n)\}_{n\in\mathbb{N}}$ and $\{c(n)\}_{n\in\mathbb{N}}$ are two sequences of positive real numbers, satisfying the following two hypotheses.
\begin{itemize}
\item[(A).] $\lim\limits_{n\to\infty} a(n) = \infty, \quad \liminf\limits_{n\to\infty} c(n) = 0$;
\item[(B).] the sequence of ratio $\{a(n)/c(n)\}_{n\in\mathbb{N}}$ is strictly increasing whenever $n>N_0$, for some integer $N_{0}>0$.
\end{itemize}
Then for every integer $N>0$, there exists an integer $n^* >N$, such that $n^*$ is both an L-index of $\{a(n)\}$ and a P-index of $\{c(n)\}$, (i.e., $\forall n>n^*, a(n) > a (n^*) \quad and \quad \forall n<n^*, c(n) > c(n^*)$).
\end{lemma}
\begin{proof}[Proof of Lemma \ref{lem-3}]
Due to $\lim\limits_{n\to\infty}a(n)=\infty$ in Hypothesis (A), it follows from Lemma \ref{lem-7.5} that there are infinitely many L-indices of $\{a(n)\}_{n\in
\mathbb{N}}$. Therefore, we are able to denote $0=n_{0}<n_1<n_2<\cdots$ as the maximal infinite L-sequence of $\{a(n)\}_{n\in\mathbb{N}}$.

Fix the integer $N_{0}$ in Hypothesis (B), and choose an integer $K>0$, such that $n_{K-1}>N_0$. Due to Hypothesis (B),  for any $n \in (n_{k-1}, n_k)$(if any) with $k>K$, we have
\begin{equation}\label{equ:ratio}
a(n)/c(n) < a(n_k)/c(n_k).
\end{equation}
Combining \eqref{equ:ratio} and Lemma \ref{lem-2}, it then follows that
\begin{equation}\label{equ-7}
c(n) = a(n)\frac{c(n)}{a(n)} > a(n_k)\frac{c(n_k)}{a(n_k)} = c(n_k), \forall n\in (n_{k-1}, n_k), \forall k>K.
\end{equation}
On the other hand, noting that $\liminf\limits_{n\to\infty} c(n) = 0$ in Hypothesis (B). Therefore, \eqref{equ-7} further implies that
\begin{equation}\label{equ:subsequence}
\liminf\limits_{k\to \infty} c(n_k) = 0,
\end{equation}
which by Lemma \ref{lem-7.5} yields that the subsequence $\{c(n_k)\}_{k\in\mathbb{N}}$ has infinitely many P-indices, and one can select a subsequence $0=n_{0}<n_{k_1}<n_{k_2}<\cdots$ from the sequence $\{n_{i}\}$ as the maximal infinite P-sequence of $\{c(n_{k})\}_{k\in\mathbb{N}}$.
			
Choose an integer $M>0$, such that $k_m>K$, whenever $m>M$. By the definition of $\{n_{k_i}\}$, every index $n_{k_m}$ with $m>M$ is an L-index of $\{a(n)\}_{n\in\mathbb{N}}$. Moreover, we further claim that
\begin{claim}
$n_{k_m}$ is also a P-index of $\{c(n)\}_{n\in\mathbb{N}}$.
\end{claim}
To see this, note that for every $n<n_{k_m}$, there exists a $m^\prime \in [0,m)$ such that $n_{k_{m^\prime}} \leq n < n_{k_{m^\prime}+1}$. Therefore,
$c(n)\geq c(n_{k_{m^\prime}+1}) \geq c(n_{k_m})$, and these two equalities cannot hold simultaneously. This implies that $c(n)>c(n_{k_m})$, and thus completes the proof of the claim.
			
Consequently, we showed that every $n_{k_m}$ with $m>M$ is both L-index of $\{a(n)\}_{n\in\mathbb{N}}$ and a P-index of $\{c(n)\}_{n\in\mathbb{N}}$. The proof of Lemma \ref{lem-3} is thus completed.
\end{proof}

\bigskip

Derived from Lemma \ref{lem-3}, we are now ready to prove Lemmas \ref{cor-1} and \ref{cor-2} as follows.
		
\begin{proof}[Proof of Lemma \ref{cor-1}]
Denote two sequences
\begin{equation*}
a_\epsilon (n):= \psi(n)/(b_{\psi}-\epsilon)^n,\quad c_\epsilon (n):= \psi(n)/(b_{\psi}+\epsilon)^n,~~\mbox{for every}~~n\in\mathbb{N}.
\end{equation*}
Divide by $b_{\psi}-\epsilon$ and $b_{\psi}+\epsilon$ respectively on both sides of the equation: $\liminf\limits_{n\to \infty} {\psi(n)}^{1/n} = b$. It then follows that
\begin{equation*}
\liminf\limits_{n\to \infty} {\psi(n)}^{1/n}/(b_{\psi}-\epsilon) = b_{\psi}/(b_{\psi}-\epsilon), \quad \liminf\limits_{n\to \infty} {\psi(n)}^{1/n}/(b_{\psi}+\epsilon) = b_{\psi}/(b_{\psi}+\epsilon).
\end{equation*}
Note that $b_{\psi}/(b_{\psi}-\epsilon)>1, 0<b_{\psi}/(b_{\psi}+\epsilon)<1$. Thus, $\lim\limits_{n\to\infty} a_\epsilon (n) = \infty, \liminf\limits_{n\to\infty} c_\epsilon (n) = 0$.

Additionally, it is also apparent that the sequence
$$
\left\{\frac{a_\epsilon (n)}{c_\epsilon (n)}\right\}_{n\in\mathbb{N}} = \left\{\left(\frac{b_{\psi}+\epsilon}{b_{\psi}-\epsilon}\right)^n\right\}_{n\in\mathbb{N}}
$$
is strictly increasing. Thus, applying Lemma \ref{lem-3} on the sequences $\{a_{\epsilon}(n)\}_{n\in\mathbb{N}}$ and $\{c_{\epsilon}(n)\}_{n\in\mathbb{N}}$, it directly implies \eqref{equ-8}, as we wanted.
\end{proof}
		
The proof of Lemma \ref{cor-2} is analogous to Lemma \ref{cor-1}.	
\begin{proof}[Proof of Lemma \ref{cor-2}]
Denote two sequences
$$
a(n) = \psi(n)/n,\quad c_\epsilon (n) = \psi(n)/(1+\epsilon)^n,~~~~\mbox{for every}~~n\in\mathbb{N}.
$$
Hypothesis $(A)$ in Theorem \ref{thm-4} implies $\lim\limits_{n\to\infty}a(n)=\infty$. Hypothesis $(B)$ in Theorem \ref{thm-4} implies that $\liminf\limits_{n\to\infty} c_\epsilon(n) = 0$. Moreover,
$$
\left\{a(n)/c_\epsilon(n)\right\}_{n\in\mathbb{N}} = \left\{(1+\epsilon)^n/n\right\}_{n\in\mathbb{N}}
$$
is strictly increasing whenever $n$ is sufficiently large. Hence by applying Lemma \ref{lem-3} on sequences $\{a(n)\}_{n\in\mathbb{N}}$ and $\{c_{\epsilon}(n)\}_{n\in\mathbb{N}}$, it implies \eqref{equ-9}, as we wanted.
		\end{proof}
		
\subsection{Proof of Theorem \ref{thm-1}}
With the aid of Lemmas \ref{cor-1} and \ref{cor-2}, we are ready to proceed the proof of Theorem \ref{thm-1}, which is split into the proof of Propositions \ref{thm-2}, \ref{thm-3}, \ref{thm-4}. Each proof of these propositions occupies a subsection, and follows from the same strategy: first, construct $g_\psi$, and subsequently verify the corresponding properties.

\subsubsection{Proof of Proposition \ref{thm-2}}
\begin{proof}[Proof of Proposition \ref{thm-2}]
By hypothesis that $\lim\limits_{n\to\infty}{\psi(n)}^{1/n}=\infty$, it follows from Lemma \ref{lem-7.5} that the sequence $\{{\psi(n)}^{1/n}\}_{n\in\mathbb{N}}$ admits a maximal infinite L-sequence $0=n_{0}<n_1 < n_2 < \cdots$. Moreover, there exists a $K$ such that ${\psi(n_k)}^{1/n_k}>1$ for all $k\geq K$.
			
\textbf{Construction of $g_{\psi}$:} We construct $g_\psi$ as follows.
\begin{itemize}
\item For $n<n_K$, define $g_\psi(n)$ recursively. Suppose $g_\psi(n+1)$ is defined, then $g_\psi(n):=\min\{g_\psi(n+1), \psi(n)\}$.
\item For $n \in [n_k, n_{k+1})$ with $k\geq K$, define $g_\psi(n):={\psi(n_k)}^{n/n_k}$.
\end{itemize}
\textbf{Properties verification:} In the rest of the proof, let us verify the desired properties.
\begin{itemize}
\item First let us show Property $(a)$. When $n<n_K$, it follows directly from our construction that $g_\psi(n)\leq \psi(n)$. When $n \in [n_k, n_{k+1}), k\geq K$, note that by the definition of L-indices, $\psi(n)^{1/n} \geq \psi(n_k)^{1/n_k}$. Hence $g_\psi(n) = \left({\psi(n_k)}^{1/n_k}\right)^n \leq \left({\psi(n)}^{1/n}\right)^n = \psi(n)$, and Property $(a)$ gets proved, as we want.
\item Next, Property $(b)$ is directly derived from $g_\psi(n_k) = \psi(n_k)$ for all $k\geq K$.
\item Regarding to showing that $g_{\psi}$ is non-decreasing and Property $(c)$, let us consider the ratio between two successive terms of $g_\psi$. When $n<n_K$, it is clear from our construction that $g_\psi(n)\leq g_\psi(n+1)$. In what follows, let us turn to the case where $n\geq n_K$. In particular, fix $n\in [n_k, n_{k+1})$, with $k\geq K$.
\begin{itemize}
\item If $n=n_k$, then
\begin{equation*}
\frac{g_\psi(n)}{g_\psi(n-1)} = {\psi(n_k)}^{1/n_k}\left(\frac{ {\psi(n_k)}^{1/n_k} }{{\psi(n_{k-1})}^{1/n_{k-1}}}\right)^{n_k-1} \geq \psi(n_k)^{1/n_k}.
\end{equation*}
\item Otherwise, we have
\begin{equation*}
\frac{g_\psi(n)}{g_\psi(n-1)} = \psi(n_k)^{1/n_k},
\end{equation*}
\end{itemize}
In both cases, we always have
\begin{equation}
g_\psi(n)/g_\psi(n-1) \geq \psi(n_k)^{1/n_k}.
\end{equation}
By the hypothesis, $\psi(n_k)^{1/n_k}\geq 1$ whenever $k\geq K$, and tends to infinity as $k\to \infty$. Consequently, $g_\psi$ is non-decreasing, and Property $(c)$ is also deduced.
\end{itemize}
Thus, we have verified Properties $(a),(b)$ and $(c)$, and the proof of Proposition \ref{thm-2} completes.
\end{proof}

\subsubsection{Proof of Proposition \ref{thm-3}}
\begin{proof}[Proof of Proposition \ref{thm-3}]
Fix a constant $\epsilon\in (0,b_{\psi}-1)$, and construct inductively a sequence of indices $\{n_j\}_{j\in\mathbb{N}}$ as follows. Choose a positive integer $n_1$, such that
\begin{equation*}
\psi(n)>(b_{\psi}-\epsilon)^{n-n_1} \psi(n_1), \forall n>n_1 \quad and \quad \psi(n) > (b_{\psi}+\epsilon)^{n-n_1} \psi(n_1), \forall n < n_1
\end{equation*}
The existence of $n_1$ is guaranteed by Lemma \ref{cor-1}. Suppose $n_1<\cdots<n_{j-1}$ is already chosen, choose a positive integer $n_j > n_{j-1}$, such that
\begin{equation}\label{equ-a}
\psi(n)>(b_{\psi}-\epsilon/j)^{n-n_j} \psi(n_j), \forall n>n_j \quad and \quad \psi(n) > (b_{\psi}+\epsilon/j)^{n-n_j} \psi(n_j), \forall n < n_j
\end{equation}
the existence of $n_j$ is also guaranteed by Lemma \ref{cor-1}. The sequence $\{n_j\}_{j\in\mathbb{N}}$ partitions natural numbers into different intervals.

\textbf{Construction of $g_{\psi}$:} We construct $g_{\psi}$ as follows.
\begin{itemize}
\item For $n=n_j$, we define $g_\psi(n_j):= \psi(n_j)$.
\item For $1\leq n<n_1$, we define $g_\psi(n)$ recursively: suppose $g_{\psi}(n+1)$ is defined then $g_\psi(n):= \min\{g_\psi(n+1),\psi(n)\}$.
\item For each $j>1$, and $n_j<n<n_{j+1}$, we study two auxiliary functions $f_{j}$ and $h_{j}$ before defining $g_\psi(n)$. Denote by
\begin{equation}\label{equ-b}
f_j(x) = \left(b_{\psi}-\frac{\epsilon}{j}\right)^{x-n_j} \psi(n_j), \quad h_j(x) = \left(b_{\psi}+\frac{\epsilon}{j+1}\right)^{x-n_{j+1}} \psi(n_{j+1}).
\end{equation}
It follows from \eqref{equ-a} and \eqref{equ-b} that
\begin{equation}\label{equ-a*(1)}
\psi(n)>f_j(n), \quad \mbox{if} \quad n \in (n_j,n_{j+1}],
\end{equation}
and
\begin{equation}\label{equ-a*(2)}
\psi(n)>h_j(n), \quad \mbox{if} \quad n \in [n_j,n_{j+1}).
\end{equation}
Moreover, it also follows from \eqref{equ-b} that
\begin{equation}\label{equ-c}
f_j(n_j) = \psi(n_j) \quad and \quad h_j(n_{j+1}) = \psi(n_{j+1}).
\end{equation}
Hence $f_j(n_j) > h_j(n_j)$ and $f_j(n_{j+1}) < h_j(n_{j+1})$. Thus there is a unique $\hat{n}_j\in [n_j, n_{j+1})$ such that $f_j(x)\geq h_j(x)$ whenever $n_j \leq n\leq \hat{n}_j$, while $f_j(\hat{n}_j + 1) < h_j(\hat{n}_j + 1)$. In fact, since
\begin{align*}
    f_j(n) &= \left(b_{\psi}-\frac{\epsilon}{j}\right)^{n-\hat{n}_j-1}\cdot f_j(\hat{n}_j+1)\\
    &\leq \left(b_{\psi}+\frac{\epsilon}{j+1}\right)^{n-\hat{n}_j-1}\cdot f_j(\hat{n}_j+1)\\
    &< \left(b_{\psi}+\frac{\epsilon}{j+1}\right)^{n-\hat{n}_j-1}\cdot h_j(\hat{n}_j+1)=h_j(n).
\end{align*}
we have $f_j(n) < h_j(n)$ once $n>\hat{n}_j$.

For $n\in (n_j, n_{j+1})$, based on $f_{j}(n)$ and $h_{j}(n)$, we define
\begin{equation}\label{equ-d}
g_\psi(n):= \max\{f_j(n), h_j(n)\} = \left\{ \begin{array}{ccc} f_j(n) & \mbox{if} & n_j<n\leq \hat{n}_j; \\ h_j(n) & \mbox{if} & \hat{n}_j < n < n_{j+1}. \end{array}\right.
\end{equation}
Since $\max\{f_j(n_j), h_j(n_j)\}=f_j(n_j)=\psi(n_j)$ and $\max\{f_j(n_{j+1}), h_j(n_{j+1})\}=h_j(n_{j+1})=\psi(n_{j+1})$, One can extend $g_{\psi}$ to $n\in [n_j, n_{j+1}]$, see Figure \ref{Figure1.3}.
\end{itemize}
\begin{figure}[htbp]
\centering
\includegraphics[width=14cm,height=14cm,keepaspectratio]{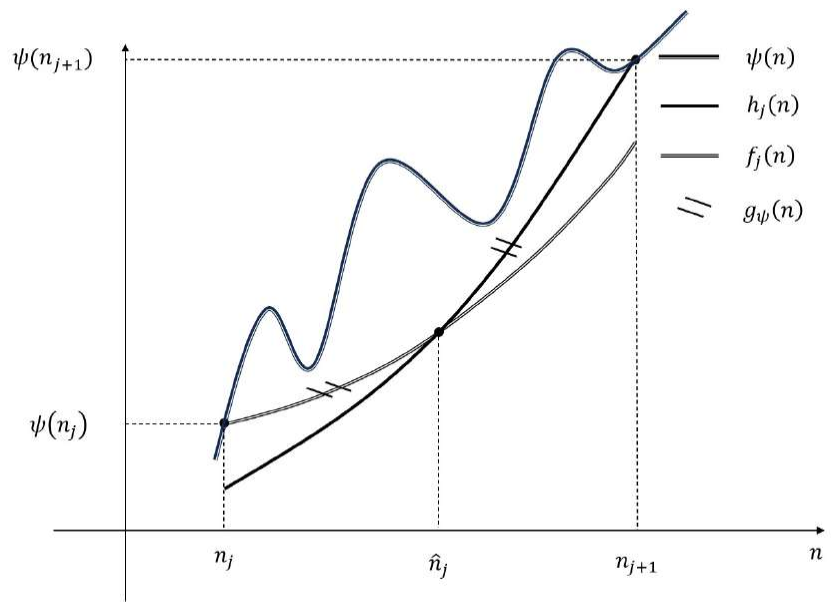}
\caption{Illustration of $g_{\psi}$ in terms of $h_{j}$ and $f_{j}$.}\label{Figure1.3}
\end{figure}
			
\textbf{Properties verification:} The properties verification of $g_{\psi}$ on $1\leq n<n_1$ is analogous to that in the proof of Proposition \ref{thm-2}, and thus we omit. In the rest of the proof, let's focus on properties verification of $g_\psi$ on a fixed interval $[n_j, n_{j+1})$, with $j\in \mathbb{N}_+$.
\begin{itemize}
\item Property $(a)$ directly follows from the fact that $g_\psi(n)\leq \psi(n)$, which is derived from \eqref{equ-a*(1)}, \eqref{equ-a*(2)}. Note that $g_\psi(n_j) = \psi(n_j)$, thus Property $(b)$ is also proved.
\item Regarding to showing $g_{\psi}$ is non-decreasing and Property $(c)$, one shall estimate $g_\psi(n+1)/g_\psi(n)$ for $n\in [n_j, n_{j+1})$. We split the estimation into three cases.
\begin{itemize}
\item ``when $n_j\leq n<\hat{n}_j$'': According to \eqref{equ-d},
\begin{equation}\label{equ_ratio1.3}
\frac{g_\psi(n+1)}{g_\psi(n)} = b-\frac{\epsilon}{j}.
\end{equation}
Since $\epsilon< b_{\psi}-1$, RHS of \eqref{equ_ratio1.3} $\geq 1$. Moreover, it tends to $b_{\psi}$ as $j\to \infty$.
\item ``when $\hat{n}_j < n < n_{j+1}$'': we have
\begin{equation}\label{equ_ratio1.32}
\frac{g_\psi(n+1)}{g_\psi(n)} = b_{\psi}+\frac{\epsilon}{j+1},
\end{equation}
with RHS of \eqref{equ_ratio1.32} $\geq1$ and tending to $b_{\psi}$ as $j\to\infty$.
\item ``when $n = \hat{n}_j$'': It can be estimated as follows: on one hand,
\begin{equation*}
\frac{g_\psi(\hat{n}_j+1)}{g_\psi(\hat{n}_j)} = \frac{h_j(\hat{n}_j+1)}{f_j(\hat{n}_j)} = \left(b_{\psi}-\frac{\epsilon}{j}\right)\cdot \frac{h_j(\hat{n}_j+1)}{f_j(\hat{n}_j+1)} > b_{\psi}-\frac{\epsilon}{j}.
\end{equation*}
On the other hand,
\begin{equation*}
\frac{g_\psi(\hat{n}_j+1)}{g_\psi(\hat{n}_j)} = \frac{h_j(\hat{n}_j+1)}{f_j(\hat{n}_j)} = \left(b_{\psi}+\frac{\epsilon}{j+1}\right)\cdot \frac{h_j(\hat{n}_j)}{f_j(\hat{n}_j)} \leq b_{\psi}+\frac{\epsilon}{j+1}.
\end{equation*}
Consequently, $g_\psi(\hat{n}_j+1)/g_\psi(\hat{n}_j) > 1$ and tends to $b_{\psi}$ as $j\to \infty$.
\end{itemize}
In summary, $g_\psi$ is non-decreasing, and $\lim\limits_{n\to\infty}g_\psi(n+1)/g_\psi(n)=b_{\psi}$, as we want.
\end{itemize}
Thus, we have verified Properties $(a),(b)$ and $(c)$, and the proof of Proposition \ref{thm-3} completes.
\end{proof}
		
\subsubsection{Proof of Proposition \ref{thm-4}}
\begin{proof}[Proof of Proposition \ref{thm-4}]
The proof is analogous to the proof of Proposition \ref{thm-3}. Choose inductively a sequence of indices $\{n_j\}_{j\in\mathbb{N}}$ as follows: fix $\epsilon>0$, and choose a positive integer $n_1$, such that
\begin{equation*}
\psi(n)>n \psi(n_1)/n_1, \forall n>n_1 \quad and \quad \psi(n) > (1+\epsilon)^{n-n_1} \psi(n_1), \forall n < n_1.
\end{equation*}
Suppose $n_1<\cdots<n_{j-1}$ is chosen, define $n_j>n_{j-1}$, such that
\begin{equation}\label{equ-metal}
\psi(n)>n \psi(n_j)/n_j, \forall n>n_j \quad and \quad \psi(n) > (1+1/n_{j-1})^{n-n_j} \psi(n_j), \forall n < n_j
\end{equation}
The existence of $\{n_j\}_{j\in\mathbb{N}}$ is guaranteed by Lemma \ref{cor-1}.

\textbf{Construction of $g_{\psi}$:} We construct $g_{\psi}$ as follows.
\begin{itemize}
\item For $n=n_j$, define $g_\psi(n_j):= \psi(n_j)$.
\item For $1\leq n<n_1$, $g_\psi(n)$ is defined inductively: suppose $g_{\psi}$ is already defined, define $g_\psi(n):= \min\{g_\psi(n+1),\psi(n)\}, 1\leq n < n_1$.
\item For each $j\geq 1$, and $n_{j}<n<n_{j+1}$, we study two auxiliary functions $f_j, h_j$. Denote by
\begin{equation}\label{equ-wood}
f_j(x):= \frac{x\psi(n_j)}{n_j} \quad \quad h_j(x):= \left(1+\frac{1}{n_j}\right)^{x-n_{j+1}} \psi(n_{j+1}).
\end{equation}
Hence, \eqref{equ-metal} implies that				
\begin{equation}\label{equ-metal(1)}
f_j(n) < \psi(n), \quad \mbox{if} \quad n\in (n_j, n_{j+1}]
\end{equation}
and
\begin{equation}\label{equ-metal(2)}
h_j(n) < \psi(n), \quad \mbox{if} \quad n\in [n_j, n_{j+1}).
\end{equation}
Moreover, it is also from \eqref{equ-wood} that
$$
f_j(n_j) = \psi(n_j) \quad\mbox{and}\quad h_j(n_{j+1}) = \psi(n_{j+1}).
$$
Combining this with \eqref{equ-metal(1)}, \eqref{equ-metal(2)}, it implies that there is $\hat{n}_j \in [n_j, n_{j+1})$, such that
\begin{equation}\label{equ-water}
f_j(n)\geq h_j(n) \quad \mbox{whenever}, \quad n_j \leq n\leq \hat{n}_j,
\end{equation}
while
\begin{equation}\label{equ-fire}
f_j(\hat{n}_j + 1) < h_j(\hat{n}_j + 1).
\end{equation}
Together with the Bernoulli's inequality\footnote{Bernoulli's inequality: for real numbers $x>-1, a\geq 1$
        \begin{equation}\label{e.Bernoulli's}
            (1+x)^a\geq 1+ax.
        \end{equation}
        },  \eqref{equ-fire} further implies that
\begin{align*}
   h_j(n)
   &=(1+\frac{1}{n_j})^{n-\hat{n}_j-1}h_j(\hat{n}_j+1)\\
   &>(1+\frac{1}{n_j})^{n-\hat{n}_j-1}f_j(\hat{n}_j+1)\\
   &=(1+\frac{1}{n_j})^{n-\hat{n}_j-1}\frac{\hat{n}_j+1}{n}f_j(n)\\
   &\geq (1+(n-\hat{n}_j-1)\frac{1}{n_j})\frac{\hat{n}_j+1}{n}f_j(n)\\
   &\geq (1+(n-\hat{n}_j-1)\frac{1}{\hat{n}_j+1})\frac{\hat{n}_j+1}{n}f_j(n)=f_j(n)\\
\end{align*}
\begin{equation}\label{equ-soil}
h_j(n) > f_j(n) \quad \text{if} \quad \hat{n}_j < n < n_{j+1}.
\end{equation}
Based on $f_{j}(n)$ and $h_{j}(n)$, define
\begin{equation}\label{equ-heaven}
g_\psi(n):= \max\{f_j(n), h_j(n)\} = \left\{ \begin{array}{ccc} f_j(n) & \mbox{if} & n_j<n\leq \hat{n}_j; \\ h_j(n) & \mbox{if} & \hat{n}_j < n < n_{j+1}, \end{array}\right.
\end{equation}
Since $\max\{f_j(n_j), h_j(n_j)\}=f_j(n_j)=\psi(n_j)$ and $\max\{f_j(n_{j+1}), h_j(n_{j+1})\}=h_j(n_{j+1})=\psi(n_{j+1})$, One can extend $g_{\psi}$ to $n\in [n_j, n_{j+1}]$,
see Figure \ref{Figure1.4}.
\end{itemize}

\begin{figure}[htbp]
\centering\includegraphics[width=14cm,height=14cm,keepaspectratio]{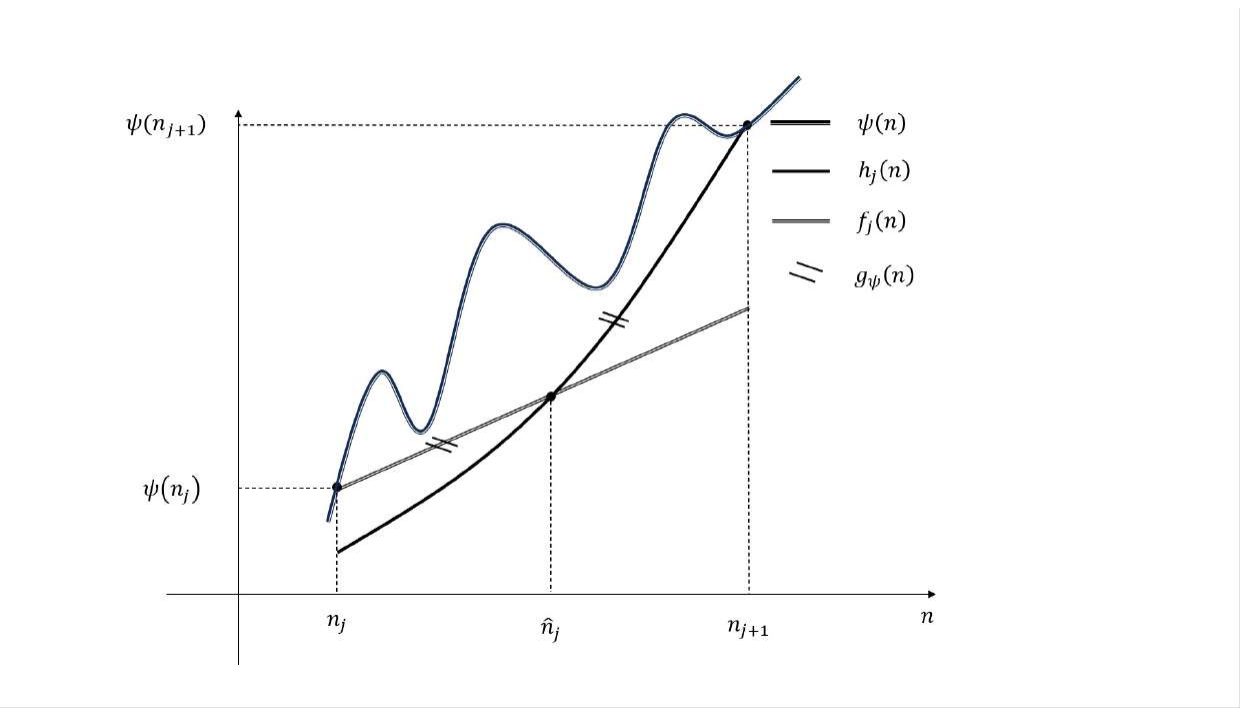}
\caption{Illustration of $g_{\psi}$ in terms of $h_{j}$ and $f_{j}$.}\label{Figure1.4}
\end{figure}

\textbf{Properties verification:}
\begin{itemize}
\item Property $(a)$ follows \eqref{equ-metal(1)} and \eqref{equ-metal(2)}.
\item Property $(b)$ follows from the observation that $g_\psi(n_j) = \psi(n_j)$.
\item Regarding to Property $(c)$, note that
\begin{equation}
\frac{g_\psi(n)}{n} \geq \frac{f_j(n)}{n} = \frac{\psi(n_j)}{n_j} \to \infty \quad as \quad j\to \infty,
\end{equation}
which derives property $(c)$.
\item Regarding to showing that $g_\psi(n)$ is indeed non-decreasing and Property $(d)$, we estimate $g_\psi(n+1)/g_\psi(n)$ for $n\in [n_j, n_{j+1})$ as follows.
\begin{itemize}
\item If $n_j\leq n<\hat{n}_j$, then
\begin{equation*}
\frac{g_\psi(n+1)}{g_\psi(n)} = \frac{n+1}{n};
\end{equation*}
\item if $\hat{n}_j<n\leq n_{j+1}$, then
\begin{equation*}
\frac{g_\psi(n+1)}{g_\psi(n)} = 1+\frac{1}{n_j}.
\end{equation*}
\end{itemize}
In both cases, we have $g_\psi(n+1)/g_\psi(n)>1$ and $g_\psi(n+1)/g_\psi(n)\to 1$ as $j\to\infty$. Eventually, for $n=\hat{n}_j$, we have
\begin{equation*}
\frac{g_\psi(\hat{n}_j+1)}{g_\psi(\hat{n}_j)} = \frac{h_j(\hat{n}_j+1)}{f_j(\hat{n}_j)} = \frac{\hat{n}_j + 1}{\hat{n}_j} \cdot \frac{h_j(\hat{n}_j+1)}{f_j(\hat{n}_j+1)} > \frac{\hat{n}_j + 1}{\hat{n}_j}
\end{equation*}
and
\begin{equation*}
\frac{g_\psi(\hat{n}_j+1)}{g_\psi(\hat{n}_j)} = \frac{h_j(\hat{n}_j+1)}{f_j(\hat{n}_j)} = \left(1+\frac{1}{n_j}\right) \cdot \frac{h_j(\hat{n}_j)}{f_j(\hat{n}_j)} \leq 1+\frac{1}{n_j}.
\end{equation*}
Thus, we have $g_\psi(n)$ is indeed non-decreasing and Property $(d)$.
\end{itemize}
Thus we have verified Properties $(a),(b),(c)$ and $(d)$, and the proof of Proposition \ref{thm-4} is completed.
\end{proof}

{\bf Acknowledgements:}
L. Fang was partially supported by NSFC (No. 11801591). C. G. Moreira was partially supported by CNPq and FRPERJ. Y. Zhang was partially supported by NSFC Nos. 12161141002, 12271432, and Guangdong Basic and Applied Basic Research Foundation No. 2024A1515010974.

\end{document}